\input ./style/arxiv-general.cfg
\documentclass[MSNbibl,number,citesort,seceqn,dvips]{arxbj}
\makeatletter
   \@ifpackageloaded{graphicx}{}{\usepackage{graphicx}}
\makeatother
\usepackage[dvips,all]{xy}
\usepackage{xypic}
\usepackage{graphicx}

\volume{23}
\issue{1}
\pubyear{2017}
\firstpage{670}
\lastpage{709}
\doi{10.3150/14-BEJ680}
\docsubty{FLA}

\makeatletter
\def\mid{|}
\newcommand{\rrvert}{\vert}
\newcommand{\rrVert}{\Vert}
\newcommand{\llvert}{\vert}
\newcommand{\llVert}{\Vert}
\newtheorem{TD}{Theorem-Definition}
\newproclaim{definition}[TD]{Definition}

\newtheorem{theorem}[TD]{Theorem}
\newtheorem{corollary}[TD]{Corollary}
\newtheorem{lemma}[TD]{Lemma}
\newtheorem{remark}[TD]{Remark}
\makeatother

\begin{document}
\begin{frontmatter}

\title{Nonasymptotic analysis of adaptive and annealed Feynman--Kac
particle models}
\runtitle{Nonasymptotic analysis of adaptive and annealed Feynman--Kac
particle models}

\begin{aug}
\author[A,B]{\inits{F.}\fnms{Fran\c{c}ois} \snm{Giraud}\corref{}\thanksref{A,B,e1}\ead[label=e1,mark]{francois.giraud@ens-cachan.org}}
\and
\author[C]{\inits{P.}\fnms{Pierre} \snm{Del Moral}\thanksref{B,e2}\ead[label=e2,mark]{pierre.del-moral@inria.fr}}

\address[A]{CEA-CESTA, 33114 Le Barp, France.}
\address[B]{INRIA Bordeaux Sud-Ouest, team ALEA, Domaine
Universitaire, 351, cours de la Lib\'{e}ration, 33405 Talence Cedex, France.
\printead{e1}}
\address[C]{School of Maths and Stats. University of New South
Wales, Sydney, Australia.\\ \printead{e2}}
\end{aug}

%
\received{\smonth{9} \syear{2012}}
%
\revised{\smonth{9} \syear{2014}}

%
\begin{abstract}
Sequential and quantum Monte Carlo methods, as well as genetic type
search algorithms
can be interpreted as a mean field and interacting particle
approximations of Feynman--Kac models in distribution spaces.
The performance of these population Monte Carlo algorithms is strongly
related to the stability properties of nonlinear Feynman--Kac
semigroups. In this paper, we analyze these models in terms of
Dobrushin ergodic coefficients of the reference Markov transitions and
the oscillations of the potential functions. Sufficient conditions for
uniform concentration inequalities w.r.t. time are expressed explicitly
in terms of these two quantities. We provide an original perturbation
analysis that applies to annealed and adaptive Feynman--Kac models,
yielding what seems to be the first results of this kind for these
types of models. Special attention is devoted to the particular case of
Boltzmann--Gibbs measures' sampling. In this context, we design an
explicit way of tuning the number of Markov chain Monte Carlo
iterations with temperature schedule. We also design an alternative
interacting particle method based on an adaptive strategy to define the
temperature increments.
The theoretical analysis of the performance of this adaptive model is
much more involved as both the potential functions and the reference
Markov transitions now depend on the random evolution on the particle
model. The nonasymptotic analysis of these complex adaptive models is
an open research problem.
We initiate this study with the concentration analysis of a simplified
adaptive models based on reference Markov transitions that coincide
with the limiting quantities, as the number of particles tends to infinity.
\end{abstract}

%
\end{frontmatter}

\section*{Introduction}\label{sec1}

Feynman--Kac (\textit{abbreviate FK}) particle methods, also called
sequential, quantum or diffusion Monte Carlo methods, are stochastic
algorithms to sample from a sequence of complex high-dimensional
probability distributions. These stochastic simulation techniques are
of current use in numerical physics \cite
{Assaraf-overview,Assaraf,Hetherington} to compute ground state
energies in molecular systems. They are also used in statistics, signal
processing and information sciences \cite
{Cappe,DM-filt,DM-D-J,DM-Guionnet-2} to compute posterior distributions
of a partially observed signal or unknown parameters. In the
evolutionary computing literature, these Monte Carlo methods are used
as natural population search algorithms for solving optimization problems.
From the pure mathematical viewpoint, these advanced Monte Carlo
methods are an interacting particle system (\textit{abbreviate IPS})
interpretation of FK models. For a more thorough discussion on these
models, we refer the reader to the monograph \cite{DM-FK}, and the
references therein.
The principle (see also \cite{DM-D-J} and the references therein) is to
approximate a sequence of target probability distributions $(\eta_n)_n$
by a large cloud of random samples termed particles or walkers. The
algorithm starts with $N$ independent samples from $\eta_0$ and then
alternates two types of steps: an acceptance-rejection scheme equipped
with a selection type recycling mechanism, and a sequence of free
exploration of the state space.

In the recycling stage, the current cloud of particles is transformed
by randomly duplicating and eliminating particles in a suitable way,
similarly to a selection step in models of population genetics. In the
Markov evolution step, particles move independently one each other
(mutation step).

This method is often used for solving sequential problems, such as
filtering (see, e.g., \cite{Cappe,Doucet-F-G,DM-Guionnet,DM-filt}). In other
interesting problems, these algorithms also turn out to be efficient to
sample from a single target measure $\eta$. In this context, the
central idea is to find a judicious interpolating sequence of measures
$(\eta_k)_{0\leq k\leq n}$ with increasing sampling complexity,
starting from some initial distribution $\eta_0$, up to the terminal
one $\eta_n=\eta$. Consecutive measures $\eta_k$ and $\eta_{k+1}$ are
sufficiently similar to allow for efficient importance sampling and/or
acceptance-rejection sampling. The sequential aspect of the approach is
then an ``artificial way'' to introduce the difficulty of sampling
gradually. In this vein, important examples are provided by annealed
models. More generally, a crucial point is that large population sizes
allow to cover several modes simultaneously. This is an advantage
compared to standard MCMC methods that are more likely to be trapped in
local modes. These sequential samplers have been used with success in
several application domains, including rare events simulation (see \cite
{Cerou-RE}), stochastic optimization and more generally
Boltzmann--Gibbs measures sampling \cite{DM-D-J}.

Up to now, IPS algorithms have been mostly analyzed using asymptotic
(i.e., when number of particles $N$ tends to infinity) techniques,
notably through fluctuation theorems and large deviation principles
(see, e.g.,
\cite{DM-DA,DM-Guionnet-3,DM-Guionnet-2,DM-L,DM-M,Kunsch,Chopin,DM-filt,Cappe} and \cite{DM-FK} for an overview).

Some nonasymptotic theorems have been recently developed \cite
{Cerou-var,DM-D-J-adapt}, but unfortunately none of them apply to
analyze annealed and adaptive FK particle models. On the other hand,
these type of nonhomogeneous IPS algorithms are of current use for
solving concrete problems arising in numerical physics and engineering
sciences (see, e.g., \cite{Bertrand,Giraud-RB,Neal,Clapp,Deutscher,Minvielle,Jasra,Schafer}). By the lack of
nonasymptotic estimates, these particle algorithms are used as natural
heuristics.

The main contribution of this article is to analyze these two classes
of time nonhomogeneous IPS models. Our approach is based on semigroup
techniques and on an original perturbation analysis to derive several
uniform estimates w.r.t. the time parameter.

More precisely, in the case of annealed type models, we estimate
explicitly the stability properties of FK semigroup in terms of the
Dobrushin ergodic coefficient of the reference Markov chain and the
oscillations of the potential functions. We combine these techniques
with nonasymptotic theorems on $L^p$-mean error bounds \cite{DM-M}
and some useful concentration inequalities \cite{DM-Hu-Wu,DM-Rio}.
Then we provide parameter tuning strategies that allow to deduce some
useful uniform concentration inequalities w.r.t. the time parameter.
These results apply to nonhomogeneous FK models associated with cooling
temperature parameters. In this situation, the sequence of measures
$\eta_n$ is associated with a nonincreasing temperature parameter. We
mention that other independent approaches, such as Whiteley's \cite
{Whiteley} or Schweizer's \cite{Schweizer}, are based on, for
example, drift conditions, hyper-boundedness, spectral gaps, or
nonasymptotic biais and variance decompositions. These approaches lead
to convergence results that may also apply to noncompact state spaces.
To our knowledge, these techniques are restricted to nonasymptotic
variance theorems and they cannot be used to derive uniform and
exponential concentration inequalities. It seems also difficult to
extend these approaches to analyze the adaptive IPS model discussed in
the present article. To solve these questions, we develop a
perturbation technique of stochastic FK semigroups. In contrast to
traditional FK semigroup, the adaptive particle scheme is now based on
random potential functions that depend on a cooling schedule adapted to
the variability and the adaptation of the random populations.

The rest of the article is organized as follows. In a preliminary
section, we recall a few essential notions related to Dobrushin
coefficients or FK semigroups. We also provide some important
nonasymptotic results we use in the further development of the
article. Section~\ref{section-analyse-generale-FK} is concerned with
the semigroup stability analysis of these models. We also provide a
couple of uniform $L^p$-deviations and concentration estimates. In
Section~\ref{section-Gibbs}, we apply these results to Boltzmann--Gibbs
models associated with a decreasing temperature schedule. In this
context, IPS algorithm can be interpreted as a sequence of interacting
simulated annealing algorithms (\textit{abbreviate ISA}). We design an
explicit way of tuning the number of Markov chain Monte Carlo
iterations with the temperature schedule. Finally, in Section~\ref{section-adapt}, we propose an alternative ISA method based on an
original adaptive strategy to design on the flow the temperature
decrements. We provide a nonasymptotic study for a simplified model,
based on a perturbation analysis. We end the article with
$L^p$-deviation estimates as well as a couple of concentration inequalities.

\section*{Statement of some results}

Feynman--Kac particle algorithms consist in evolving a particle system
$\zeta_n =  ( \zeta_n^1, \ldots, \zeta_n^N  )$ of size~$N$, on
a given state space $E$. Their interacting evolution is decomposed into
two genetic type transitions: a selection step, associated with some
positive potential function $G_n$; and a mutation step, where the
selected particles evolve randomly according to a given Markov
transition $M_n$ (a~more detailed description of these IPS algorithms
is provided in Section~\ref{algo-classique}). In this context, the
occupation measures $ { \eta_n^N:= \frac{1}{N}\sum_{1\leq
i\leq N}\delta_{\zeta^i_n} }$ are $N$-approximations of a sequence of
measures $\eta_n$ defined by the FK recursive formulae:
\[
\eta_n(f) = \frac{\eta_{n-1}  ( G_n \times M_n.f  )}{\eta
_{n-1}(G_n)},
\]
for all bounded measurable function $f$ on $E$ (a more detailed
discussion on these evolution equations is provided in Section~\ref{section-flot-FK}).

To describe with some precision the main results of the article, we
consider the pair of parameters $(g_n,b_n)$ defined below.
\[
g_n:= \sup_{x,y\in E} \frac{G_n(x)}{G_n(y)} \quad\mbox{and}\quad
b_n = \beta(M_n):= \mathop{\mathop{
\sup_{x,y\in E}}}_{{A
\subset E}} \bigl\llvert M_n(x,A) -
M_n(y,A) \bigr\rrvert.
\]
The quantity $\beta(M_n)$ is called the Dobrushin ergodic coefficient
of the Markov transition $M_n$. One of our first main results can be
basically stated as follows.

\begin{theorem} \label{theo-statement-unif}
We assume that
\[
\sup_{p \geq1} g_p \leq M \quad\mbox{and}\quad
\sup_{p \geq1} b_p \leq\frac{a}{M(1+a)}
\]
for some finite constant $M < \infty$ and some $a \in(0,1)$. In this
situation, for any $n \geq0$, $N \geq1$, $y \geq0$ and $f \in
\mathcal{O}_1 (E)$, the probability of the event
\[
\eta_n^N(f) - \eta_n(f) \leq
\frac{r_1^{\star}  (1+h_0(y)  )
+ r_2^{\star} \sqrt{Ny} }{N}
\]
is greater than $1-\mathrm{e}^{-y}$, where $r_1^{\star}$ and $r_2^{\star}$ are
some constants that are explicitly defined in terms of $(a,M)$, and
$h_0(y) = 2y + 2 \sqrt{y}$.
\end{theorem}

In Section~\ref{section-result-unif}, under the same assumptions of
Theorem~\ref{theo-statement-unif}, we also prove uniform $L^p$-mean
error bounds as well as new concentration inequalities for unnormalized
particle models. We also extend the analysis to the situation where
$g_n \mathop{\longrightarrow}\limits_{n \rightarrow+ \infty} 1$.

We already mentioned that the regularity conditions on $b_n$ may appear
difficult to check since the Markov kernels are often dictated by the
application under study. However, we can deal with this problem as soon
as we can simulate a Markov kernel $K_n$ such that $\eta_n .  K_n = \eta
_n$. Indeed, to stabilize the system, the designer can ``add'' several
MCMC evolution steps next to each $M_n$-mutation step. From a more
formal viewpoint, the target sequence $(\eta_n)_n$ is clearly also
solution of the FK measure-valued equations associated with the Markov
kernels $M_n^{\prime} = M_n. K_n^{m_n}$, where iteration numbers $m_n$
are to be chosen loosely. This system is more stable since the
corresponding $b_n^{\prime}$ satisfy
\[
b_n^{\prime} = \beta\bigl(M_n^{\prime}\bigr)
\leq b_n .  \beta\bigl(K_n^{m_n}\bigr) \leq
b_n .  \beta(K_n)^{m_n}.
\]
In such cases, Theorem~\ref{theo-statement-unif} and its extension
provide sufficient conditions on the iteration numbers $m_n$ to ensure
the convergence and the stability properties of the algorithm.

These results apply to stochastic optimization problems. Let $V \dvtx E
\rightarrow\mathbb{R}$ be a bounded potential function, $\beta_n$ a
sequence which tends to infinity, and $m$ a reference measure on $E$.
It is well known that the sequence of Boltzmann--Gibbs measures
\[
\eta_n (\mathrm{d}x) \propto \mathrm{e}^{- \beta_n .  V(x)} m(\mathrm{d}x)
\]
concentrates on $V$'s global minima (in the sense of $m$-$\operatorname
{essinf}(V)$). In the above display, $\propto$ stands for the
proportional sign. One central observation is that these measures can
be interpreted as a FK flow of measures associated with potential
functions $G_n = \mathrm{e}^{-(\beta_n - \beta_{n-1}).V}$ and Markov kernels
$M_n = K_{\beta_n}^{m_n k_0}$ where $K_{\beta_n}$ is a simulating
annealing kernel (see Section~\ref{section-Gibbs-tuning}) and $m_n$ and
$k_0$ are given iteration parameters. In the further development of
this section, we let $K$ be the proposal transition of the simulated
annealing transition $K_{\beta}$. In this context, the IPS methods can
be used to minimize $V$. The conditions on $b_n$ and $g_n$ can be
turned into conditions on the temperature schedule $\beta_n$ and the
number of MCMC iterations $m_n$. Moreover, combining our results with
standard concentration properties of Boltzmann--Gibbs measures, we
derive some convergence results in terms of optimization performance.
In this notation, our second main result is basically stated as follows.

\begin{theorem} \label{theo-statement-optim}
Let us fix $a \in(0,1)$. We assume that for any $x \in E$, $K^{k_0}
(x,  \cdot ) \geq\delta\nu( \cdot )$ for some measure $\nu$ on $E$, some
$\delta>0$ and some $k_0 \geq1$. We also assume that the temperature
increments $\Delta_p:= \beta_p - \beta_{p-1}$ and the iteration
numbers $m_p$ satisfy the following conditions:
\[
\sup_{p \geq1} \Delta_p \leq\Delta\quad\mbox {and}\quad
m_p \geq\frac{\log({\mathrm{e}^{\Delta.  \operatorname
{osc}(V)}(1+a)}/{a})\mathrm{e}^{\operatorname{osc}(V). \beta_p}}{\delta}
\]
for some constant $\Delta$. For all $\varepsilon>0$, let $p_n^N
(\varepsilon)$ be the proportion of particles $(\zeta_n^i)$ so that $ V
(\zeta_n^i) \geq V_{\min} + \varepsilon$. Then, for any $n \geq0$, $N
\geq1$, $y \geq0$ and for all $\varepsilon^{\prime} < \varepsilon$,
the probability of the event
\[
p_n^N (\varepsilon) \leq\frac{\mathrm{e}^{-\beta_n(\varepsilon- \varepsilon
^{\prime}) }}{m_{\varepsilon^{\prime}}} +
\frac{r_1^{\star}
(1+h_0(y)  ) + r_2^{\star} \sqrt{Ny} }{N}
\]
is greater than $1-\mathrm{e}^{-y}$, with $ m_{\varepsilon^{\prime}} = m  (
V \leq V_{\min} + \varepsilon^{\prime}  )$, $h_0(y) = 2y + 2 \sqrt
{y}$, and the same constants $(r_1^{\star},r_2^{\star})$ as the ones
stated in Theorem~\ref{theo-statement-unif} (with $M=\mathrm{e}^{\Delta.  \operatorname{osc}(V)}$).
\end{theorem}

It is instructive to compare the estimates in the above theorem with
the performance analysis of the traditional simulated annealing model
(\textit{abbreviate SA}). Firstly, most of the literature on SA models is
concerned with the weak convergence of the law of the random states of
the algorithm. When the initial temperature of the scheme is greater
than some critical value, using a logarithmic cooling schedule, it is
well known that the probability for the random state to be in the
global extrema levels tends to $1$, as the time parameter tends to
$\infty$. The cooling schedule presented in Theorem~\ref
{theo-statement-optim} is again a logarithmic one. In contrast to the
SA model (see, e.g., Theorem~1 in  \cite{hajek}, and Theorem~4.3.16
in \cite{Bartoli}), Theorem~\ref{theo-statement-optim} allows to
quantify the performance analysis of the ISA model in terms of uniform
concentration inequalities that does not depend on a critical
parameter.

In practice, choosing the sequence of increments $\Delta_n = (\beta_n -
\beta_{n-1})$ in advance can cause computational problems. To solve
this problem, adaptive strategies, where increment $\Delta_n$ depends
on the current set of particles $\zeta_{n-1}$, are of common use in the
engineering community (see, e.g., \cite{Jasra,Schafer,Clapp,Deutscher,Minvielle}). In this context, we propose to study the
case where the increment $\Delta_n^N$ is chosen so that
\[
\eta_{n-1}^N \bigl(\mathrm{e}^{-{\Delta}_{n}^N  \cdot  V}\bigr) = \varepsilon,
\]
where $\varepsilon>0$ is a given constant (see Section~\ref{description-algo-adapt} for a detailed description of the algorithm).
Computationally speaking, $\varepsilon$ is the expectation of the
proportion of particles which are not concerned with the recycling
mechanism in the selection step. We interpret this particle process as
a perturbation of a theoretical FK sequence $\eta_n$ associated with a
theoretical temperature schedule~$\beta_n$. As we mentioned in the
abstract and the \hyperref[sec1]{Introduction},
the theoretical analysis of this class of adaptive particle model is
much more involved as the potential functions and the mutations
transitions now depend on the evolution of the particle populations. We
consider a simplified adaptive model associated with the limiting
reference Markov transitions. A precise description of these
adaptive particle algorithms, and their reduced simplified versions are
provided in Sections \ref{FK-representation-adapt}, \ref
{description-algo-adapt} and~\ref{sec43}.

Our main result is the following $L^p$-mean error estimate, where $(\eta
_n^N)$ denotes the sequence of empirical measures associated with the
particle system described in (\ref{loi-algo-adapt-simplif})
page \pageref{loi-algo-adapt-simplif}.

\begin{theorem} \label{theo-statement-adapt}
For any $p \geq1$, $n \geq0$, $N \geq1$ and any bounded by $1$
function $f$, we have
\[
\mathbb{E} \bigl( \bigl\llvert \eta_n^N(f) -
\eta_n(f) \bigr\rrvert ^p \bigr)^{1/p} \leq
\frac{B_p}{\sqrt{N}} \sum_{k=0}^n \prod
_{i=k+1}^n \bigl( b_i
g_i (1+c_i) \bigr),
\]
with $ { c_n = \frac{ V_{\max} \mathrm{e}^{{\Delta}_n V_{\max}}
}{\varepsilon \cdot \eta_{n-1}(V)} }$, $\Delta_n = \beta_n - \beta
_{n-1}$ and $B_p$ defined below.
%
\begin{equation}
\label{def-Bp} B_{2p}^{2p} = \frac{(2p)!}{2^p .  p!};\qquad
B_{2p+1}^{2p+1} = \frac{(2p+1)!}{2^p .  p! \sqrt{2p+1} } .  \end{equation}
\end{theorem}

Under appropriate regularity conditions on the parameters $b_n, g_n,
c_n$, we mention that these $L^p$-mean error bounds also provide
uniform concentration inequalities.

The proofs of Theorems \ref{theo-statement-unif}, \ref
{theo-statement-optim}, \ref{theo-statement-adapt} and related
uniform exponential estimates are detailed, respectively, in
Sections \ref{section-result-unif}, \ref{section-Gibbs-tuning} and
\ref{gros-result-adapt}.


\section{Some preliminaries}
\label{sec:1}

\subsection{Basic notation}
Let $(E,r)$ be a complete, separable metric space and let $\mathcal{E}$
be the $\sigma$-algebra of Borel subsets of~$E$. Denote by $\mathcal
{P}(E)$ the space of probability measures on $E$. Let $\mathcal{B}(E)$
be the space of bounded, measurable, real-valued functions on $E$.

If $\mu\in\mathcal{P}(E)$, $f \in\mathcal{B}(E)$ and $K, K_1, K_2$
are Markov kernels on $E$, then $\mu(f)$ denotes the quantity $\int_E
f(x) \mu(\mathrm{d}x)$, $K_1 .  K_2$ denotes the Markov kernel defined by
\[
(K_1.  K_2) (x,A) = \int_E
K_1(x,\mathrm{d}y) K_2(y,A),
\]
$K.  f$ denotes the function defined by
\[
K.  f(x) = \int_E K(x,\mathrm{d}y) f(y)
\]
and $\mu. K$ denotes the probability measure defined by
\[
\mu.  K(A) = \int_E K(x,A) \mu(\mathrm{d}x).
\]
If $G$ is a positive, bounded function on $E$, then $\psi_G\dvtx \mathcal
{P}(E) \rightarrow\mathcal{P}(E) $ denotes the Boltzmann--Gibbs
transformation associated with $G$, defined by
\[
\forall\mu\in\mathcal{P}(E), \forall f \in\mathcal{B}(E) \qquad \psi_G(
\mu) (f) = \frac{\mu(G \times f)}{\mu(G)}.
\]
For any $ f \in\mathcal{B}(E)$, let $ { \Vert f \Vert
_{\infty} = \sup _{x \in E}\vert f(x) \vert}$ and $\operatorname{osc}(f) =
(f_{\max}-f_{\min})$. Let $\mathcal{B}_1(E) \subset\mathcal{B}(E)$ be
the subset of functions $f$ so that $\Vert f \Vert_{\infty} \leq1$,
and $\mathcal{O}_1 (E) \subset\mathcal{B}(E) $ be the subset of
functions $f$ so that $\operatorname{osc}(f) \leq1$.
For any random variable $X\dvtx \Omega\rightarrow\mathbb{R}$ defined on
some probability space $(\Omega,\mathcal{F},\mathbb{P})$, and any $p
\geq1$, $\Vert X \Vert_{p}$ stands for the $L^p$ norm $ {
\mathbb{E}  ( |X|^p  )^{1/p} }$.
Let $\mathcal{P}_{\Omega}(E)$ be the set of random probability measures
on $E$. For all $p \geq1$, we denote by $d_p$ the distance on $\mathcal
{P}_{\Omega}(E)$ defined for all random measures $ \hat{\mu}, \hat{\nu
}$ by
\[
d_p(\hat{\mu}, \hat{\nu}) =\sup _{f\in\mathcal{O}_1 (E) } \bigl\Vert\hat{
\mu}(f) - \hat{\nu}(f) \bigr\Vert_p.
\]
Finally, for any $x\in E$, $\delta_x$ stands for the Dirac measure
centered on $x$.

\subsection{Dobrushin Ergodic coefficient} \label{section-Dob-var-tot}

Let us recall here the definitions as well as some simple properties
that will be useful in the following.

\begin{definition}
Let $\mu, \nu\in\mathcal{P}(E)$. The total variation distance
between $\mu$ and $\nu$ is defined by
\[
\llVert \mu- \nu\rrVert _{\mathrm{tv}} = \sup\bigl\{ \bigl\llvert \mu(A) -
\nu(A) \bigr\rrvert; A \in\mathcal{E} \bigr\},
\]
or in an equivalent way
\[
\llVert \mu- \nu\rrVert _{\mathrm{tv}} = \sup\bigl\{ \bigl\llvert \mu(f) -
\nu(f) \bigr\rrvert; f \in\mathcal{B}_1(E), f \geq0 \bigr\}.
\]
\end{definition}

\begin{definition}
To each Markov kernel $K$ on $E$, is associated its Dobrushin ergodic
coefficient $\beta(K) \in[0,1]$ defined by
\[
\beta(K)= \sup\bigl\{ K(x,A) - K(y,A); x,y\in E, A\in\mathcal{E} \bigr\},
\]
or in an equivalent way
\[
\beta(K)= \sup\biggl\{ \frac{\llVert  \mu.  K - \nu.  K \rrVert _{\mathrm{tv}}} {\llVert
\mu- \nu\rrVert _{\mathrm{tv}}}; \mu, \nu\in\mathcal{P}(E), \mu\neq
\nu \biggr\}.
\]
\end{definition}

The parameter $\beta(K)$ characterizes mixing properties of
the Markov kernel $K$. Note that function $\beta$ is an operator norm,
in the sense that $\beta(K_1.  K_2) \leq\beta(K_1). \beta(K_2)$, for any
couple of Markov kernels $K_1$, $K_2$. By definition, for any measures
$\mu, \nu\in\mathcal{P}(E)$ and any Markov kernel $K$, we have $\llVert  \mu.  K - \nu.  K \rrVert _{\mathrm{tv}} \leq\beta(K). \llVert  \mu- \nu\rrVert _{\mathrm{tv}}$. Otherwise, for any function $f \in\mathcal{B}(E)$,
%
\begin{equation}
\label{prop-Dob} \operatorname{osc}(K.  f) \leq\beta(K) \cdot  \operatorname{osc}(f).
\end{equation}
Further details on these ergodic coefficients can be found in the
monograph \cite{DM-FK}. Let us give here a lemma that we will need
hereinafter (a proof is given in the \hyperref[app]{Appendix}, page \pageref{preuve-BG-tv}).

\begin{lemma} \label{BG-tv}
Let $\mu, \nu\in\mathcal{P}(E)$ and $G$ a positive, bounded function
on $E$ satisfying $\sup_{x,y\in E} \frac{G(x)}{G(y)} \leq g
$, for some finite constant $g\geq0$. In this situation, we have
\[
\bigl\llVert \Psi_G(\mu) - \Psi_G(\nu) \bigr\rrVert
_{\mathrm{tv}} \leq g. \llVert \mu- \nu\rrVert _{\mathrm{tv}}.
\]
\end{lemma}

\subsection{Feynman--Kac models}

We recall here some standard tools related to FK models. They provide
useful theoretical background and notation to formalize and analyze IPS
methods (see, e.g., \cite{DM-Guionnet-2,DM-Hu-Wu,DM-M} for further details).

\subsubsection{Evolution equations} \label{section-flot-FK}
Consider a sequence of probability measures $(\eta_n)_n$, defined by an
initial measure $\eta_0$ and recursive relations
%
\begin{equation}
\label{defeta} \forall f \in\mathcal{B}(E) \qquad\eta_n(f) =
\frac{\eta_{n-1}
( G_n \times M_n.f  )}{\eta_{n-1}(G_n)}
\end{equation}
for positive functions $G_n \in\mathcal{B}(E)$ and Markov kernels
$M_n$ with $M_n(x, \cdot ) \in\mathcal{P}(E)$ and $M_n( \cdot,A) \in
\mathcal{B}_1(E)$. This is the sequence of measures we mainly wish to
approximate with the IPS algorithm. In an equivalent way, $(\eta_n)_n$
can be defined by the relation
\[
\eta_n = \phi_n(\eta_{n-1}),
\]
where $\phi_n\dvtx \mathcal{P}(E) \rightarrow\mathcal{P}(E)$ is the FK
transformation associated with potential function $G_n$ and Markov
kernel $M_n$ and defined by
\[
\phi_n(\eta_{n-1})=\psi_{G_{n}}(
\eta_{n-1}).M_n
\]
with
\[
\psi_{G_{n}}(\eta_{n-1}) (\mathrm{d}x):=\frac{1}{\eta_{n-1}(G_{n})}
G_{n}(x) \eta _{n-1}(\mathrm{d}x).
\]
The next formula provides an interpretation of the Boltzmann--Gibbs
transformation in terms of
a nonlinear Markov transport equation
\[
\Psi_{G_{n}}(\eta_{n-1}) (\mathrm{d}y)= (\eta_{n-1}
S_{n,\eta_{n-1}} ) (\mathrm{d}y):=\int\eta_{n-1}(\mathrm{d}x) S_{n,\eta_{n-1}}(x,\mathrm{d}y)
\]
with the Markov transition $ S_{n,\eta_n}$ defined below
\[
S_{n,\eta_{n-1}}(x,\mathrm{d}y)=\varepsilon_n.G_n(x)
\delta_x(\mathrm{d}y)+ \bigl(1-\varepsilon_n.G_n(x)
\bigr) \Psi_{G_{n}}(\eta_{n-1}) (\mathrm{d}y)
\]
(for any constant $\varepsilon_n > 0$ so that $\varepsilon_n.G_n \leq
1$). This implies
%
\begin{equation}
\label{rec-eta-noyau} \eta_{n}=\eta_{n-1} K_{n,\eta_{n-1}} \qquad\mbox{with } K_{n,\eta
_{n-1}}= S_{n,\eta_{n-1}}M_{n}.
\end{equation}
Therefore, $\eta_n$ can be interpreted as the distributions of the
random states $\overline{X}_n$ of a Markov chain whose Markov transitions
%
\begin{equation}
\label{eta-interp-Markov} \mathbb{P} (\overline{X}_{n+1}\in \mathrm{d}y |
\overline{X}_n=x ):=K_{n+1,\eta_n}(x,\mathrm{d}y)
\end{equation}
depend on the current distribution $\eta_n=\operatorname{Law}(\overline
{X}_n )$.

We finally recall that the measures $\eta_n$ admit the following
functional representations:
\[
\eta_n (f) = \frac{\gamma_n(f)}{\gamma_n(1)}
\]
($1$ stands for the unit function) with the unnormalized FK measures
$\gamma_n$ defined by the formulae
%
\begin{equation}
\label{defgamma} \gamma_0 = \eta_0;\qquad
\gamma_n(f) = \gamma_{n-1} (G_n \times
M_n.f).
\end{equation}
Comparing this definition with (\ref{defeta}), it is clear that the
normalizing constant $\gamma_n(1)$ satisfies
%
\begin{equation}
\label{defgamma1} \gamma_n(1) = \prod_{p=1}^n
\eta_{p-1}(G_p).
\end{equation}
The special interest given to this quantity will be motivated in
Section~\ref{section-Gibbs-motiv}.

\subsubsection{Feynman--Kac semigroup} \label{semigroupFK}

An important point is that the semigroup transformations
\[
\phi_{p,n}:=\phi_n\circ\phi_{n-1}\circ\cdots\circ
\phi_{p+1}
\]
admit a comparable structure as each of the $\phi_k$. To be more
precise, for each integer $p$, let us define the unnormalized integral
operator $Q_p$
%
\begin{equation}
\label{defQ} \forall f \in\mathcal{B}(E) \qquad Q_p.f =
G_p.M_p.f
\end{equation}
and the composition operators $Q_{p,n}$ defined by the backward recursion
%
\begin{equation}
\label{defQpn} Q_{p,n} = Q_{p+1}.  ( Q_{p+2} \cdots
Q_n ) =Q_{p+1}.  Q_{p+1,n}.
\end{equation}
We use the convention $Q_{n,n} = \mathrm{Id}$ for $p=n$. Comparing these
definitions with (\ref{defgamma}), it is clear that $\gamma_n = \gamma
_{n-1} .  Q_n$ and more generally
\[
\gamma_n = \gamma_p .  Q_{p,n}.
\]
for any $p\leq n$. The semigroup $\phi_{p,n}$ can be expressed in terms
of $Q_{p,n}$ with the following formulae:
\[
\phi_{p,n}(\mu) (f) = \frac{\mu(Q_{p,n}.f)}{\mu(Q_{p,n}.  1)}
\]
for any $f \in\mathcal{B} (E)$ and $\mu\in\mathcal{P} (E)$. Finally,
if we set
\[
P_{p,n}.f = \frac{Q_{p,n}.f}{Q_{p,n}.  1} \quad\mbox{and} \quad G_{p,n} =
Q_{p,n}.  1
\]
then we find that
\[
\phi_{p,n}(\mu) (f) = \frac{\mu(G_{p,n}.P_{p,n}.f)}{\mu(G_{p,n})},
\]
or in other words: $ \phi_{p,n}(\mu) = \psi_{G_{p,n}}(\mu).P_{p,n} $.


\subsection{The interacting particle system model}
\label{algo-classique}

The central idea is to approximate the measures $\eta_n$ by simulating
an interacting particle system $(\zeta_n)_n =  ( \zeta_n^1, \ldots,
\zeta_n^N  )_n$ of size $N$ so that
\[
\eta^N_{n}=\frac{1}{N}\sum
_{1\leq i\leq N}\delta_{\zeta^i_n}\mathop{\longrightarrow} _{N\uparrow\infty}
\eta_n.
\]
Of course, the main issue is to make precise and to quantify this
convergence. The particle model is defined as follows.

We start with $N$ independent samples $\zeta_0 = (\zeta_0^1, \ldots,
\zeta_0^N)$ from $\eta_0$. Then the particle dynamics alternates two
genetic type transitions.

During the first step, every particle $\zeta^i_{n}$ evolves to a new
particle $\widehat{\zeta}^i_{n}$ randomly chosen with the distribution
\[
S_{n+1,\eta^N_{n}}\bigl( \zeta^i_{n},\mathrm{d}x\bigr):=
\varepsilon_{n+1}.G_{n+1}\bigl( \zeta^i_{n}
\bigr) \delta_{ \zeta^i_{n}}(\mathrm{d}x)+ \bigl(1-\varepsilon_{n+1}.G_{n+1}
\bigl( \zeta^i_{n}\bigr) \bigr) \Psi _{G_{n+1}}\bigl(
\eta^N_n\bigr) (\mathrm{d}x)
\]
with the updated measures
\[
\Psi_{G_{n+1}}\bigl(\eta^N_n\bigr)=\sum
_{j=1}^N\frac{G_{n+1}( \zeta^j_{n})}{\sum_{k=1}^N G_{n+1}( \zeta^k_{n})} \delta_{ \zeta^j_{n}}.
\]
This transition can be interpreted as an acceptance-rejection scheme
with a recycling mechanism. In the second step, the selected particles
$\widehat{\zeta}^i_{n}$ evolve randomly according to the Markov
transitions $M_{n+1}$. In other words, for any $1\leq i\leq N$, we sample
a random state $ \zeta^i_{n+1}$ with distribution $M_{n+1} (\widehat
{\zeta}^i_{n},\mathrm{d}x )$.

In view of (\ref{eta-interp-Markov}), if we replace $\eta_n^N$ by $\eta
_n$, then $\zeta_n$ coincide with $N$ independent copies of the Markov
chains $\overline{X}_n$ defined in (\ref{rec-eta-noyau}). On the other
hand, by the law of large numbers, we have $\eta_0^N \simeq\eta_0$ so that
\[
\eta_1^N \simeq\eta_0^N .  K_{1,\eta_0^N} \simeq\eta_0 .  K_{1,\eta_0} =
\eta_1.
\]
Iterating this approximation procedure, the empirical measure $\eta
^N_{n}$ is expected to approximate $\eta_n$ at any time $n \geq0$. As
for the unnormalized measures $\gamma_n$, we define
\[
\gamma_n^N (1) = \prod_{p=1}^n
\eta_{p-1}^N (G_p)
\]
(mimicking formula (\ref{defgamma1})) and more generally $
{ \gamma_n^N (f) = \eta_n^N(f) \times\prod_{p=1}^n \eta_{p-1}^N (G_p)
}$. Let us mention (see, e.g., \cite{DM-filt}) that these
particle models provide an unbiased estimate of the unnormalized
measures; that is we have that
\[
\forall f \in\mathcal{B}(E)\qquad \mathbb{E} \bigl( \gamma_n^N
(f) \bigr) = \gamma_n(f).
\]
In addition to the analysis of $\eta_n^N$'s convergence, the
concentration properties of the unbiased estimators $\gamma_n^N(1)$
around their limiting values $\gamma_n(1)$ will also be considered thereafter.


\subsection{Some nonasymptotic results}
\label{section-Dob-estim}

To quantify the FK semigroup stability properties, it is convenient to
introduce the following parameters.

\begin{definition}
For any integers $p<n$, we set
\begin{eqnarray*}
b_n &:=& \beta(M_n) \quad\mbox{and}\quad b_{p,n}:=
\beta(P_{p,n}),
\\
g_n &:= &\sup_{x,y \in E}\frac{G_n(x)}{G_n(y)} \quad\mbox{and}\quad
g_{p,n}:= \sup_{x,y \in E}\frac{G_{p,n}(x)}{G_{p,n}(y)}.
\end{eqnarray*}
\end{definition}

The quantities $g_{p,n}$, and, respectively, $b_{p,n}$, reflect
the oscillations of the potential functions $G_{p,n}$, and, respectively,
the mixing properties of the
Markov transition $P_{p,n}$ associated with the FK semigroup $\phi
_{p,n}$ described in Section~\ref{semigroupFK}. Several contraction
inequalities of $\phi_{p,n}$
w.r.t. the total variation norm or different types of relative
entropies can be derived in terms of these two quantities (see, for
instance, \cite{DM-FK}).

Let
%
\begin{eqnarray}
\label{def-fonctions-h} %
h_0 &:=& x
\mapsto2(x+\sqrt{x})\quad  \mbox{and}
\nonumber
\\[-8pt]
\\[-8pt]
\nonumber
h_1 &:=& x \mapsto\frac{x}{3} + \sqrt{2x}.
\end{eqnarray}
The performance analysis developed in Sections \ref{estim-generales}
and \ref{section-Gibbs-tuning} is partly based on the three
nonasymptotic inequalities presented below.

Firstly, the following $L^p$-mean error bound for all $f \in\mathcal
{B}_1 (E)$ is inspired by \cite{DM-M} (see Section~2.2.2):
%
\begin{equation}
\mathbb{E} \bigl( \bigl\llvert \eta_n^N (f) -
\eta_n (f) \bigr\rrvert ^p \bigr)^{1/p} \leq
\frac{B_p}{\sqrt{N}} \sum_{k=0}^{n}
g_{k,n}b_{k,n}, \label{born-Lp-DMM}
\end{equation}
where $B_p$ are the constants introduced in (\ref{def-Bp}), page \pageref
{def-Bp}. A detailed proof of this estimate is given in the
\hyperref[app]{Appendix},
page \pageref{preuve-born-Lp-DMM}.

 Secondly, the following concentration inequality is derived
in \cite{DM-Rio} (see Theorem~1.2 and Section A3). For all $f \in
\mathcal{O}_1(E)$ and any $y \geq0$, we have
%
\begin{equation}
\mathbb{P} \biggl( \eta_n^N(f) - \eta_n(f)
\geq\frac{r_n}{N} \bigl( 1+h_0 (y) \bigr) + {\overline{
\beta}_n} \sqrt{\frac{2y}{N}} \biggr) \leq \mathrm{e}^{-y},
\label{ineg-conc}
\end{equation}
where $r_n$ and $\overline{\beta}_n$ are constants so that
\[
\cases{ %
\displaystyle r_n \leq 4 \sum
_{p=0}^{n} g_{p,n}^3
b_{p,n},
\vspace*{2pt}\cr
\displaystyle\overline{\beta}^2_n  \leq 4 \sum
_{p=0}^{n} g_{p,n}^2
b_{p,n}^2.}
\]
To be more precise with the derivation of these estimates, the upper
bound on $r_n$ is based on the quadratic remainder estimates provided
in the end of page 14 and
on page 28 (line 17) in \cite{DM-Rio}. The function $h_0$ coincides
with the function $\epsilon_0^{-1}$ used in \cite{DM-Rio}, and we have
used the estimate provided on page 16 (line 13). The upper bound on $
\overline{\beta}^2_n$ is a direct consequence of the estimates provided
on page 9 (line 2) and on page 28 (line 13) in \cite{DM-Rio}.

Thirdly, the following concentration inequality for
unnormalized particle models $\gamma_n^N$ is provided in \cite
{DM-Hu-Wu} (see Theorem $6.13$).

$\forall\epsilon\in\{ +1,-1 \}$ and $\forall y \geq0$:
%
\begin{equation}
\label{ineg-conc-nor} \mathbb{P} \biggl( \frac{\epsilon}{n} \log \biggl(
\frac{\gamma_n^N
(1)}{\gamma_n(1)} \biggr) \geq\frac{\bar{r} (n) }{N} h_0 (y) + \tau
_n^{\star} \bar{\sigma}_n^2
h_1 \biggl( \frac{y}{N.\bar{\sigma}_n^2} \biggr) \biggr) \leq \mathrm{e}^{-y},
\end{equation}
where quantities $\tau_n^{\star}$, $\bar{\sigma}_n^2$ and
$\bar{r} (n)$ can be estimated this way:
\begin{itemize}
\item$ { \tau_n^{\star} = \sup_{0 \leq q \leq n}
\tau_{q,n} }$, where $\tau_{q,n}$ satisfy
%
\begin{equation}
\label{estim-tau} \tau_{q,n} \leq\frac{4}{n} \sum
_{p=q}^{n-1} g_{q,p} .  g_{p+1} .  b_{q,p};
\end{equation}

\item$ { \bar{\sigma}_n^2 = \sum_{q=0}^{n-1} \sigma_q^2
 ( \frac{\tau_{q,n} }{\tau_n^{\star} }  )^2 }$ where $\sigma
_q$ satisfy $\sigma_q \leq1$;

\item$\bar{r} (n)$ satisfy
%
\begin{equation}
\label{estim-r} \bar{r} (n) \leq\frac{8}{n} \sum
_{0 \leq q \leq p <n} g_{p+1} .  g_{q,p}^3
.b_{q,p}.
\end{equation}
\end{itemize}


\section{Nonasymptotic theorems} \label{section-analyse-generale-FK}

The formulae (\ref{born-Lp-DMM}), (\ref{ineg-conc}) and (\ref
{ineg-conc-nor}) provide explicit nonasymptotic estimates in terms of
the quantities $g_{p,n}$ and $b_{p,n}$. Written this way, they hardly
apply to any IPS parameters tuning decision, since the only known or
calculable objects are generally the reference Markov chain $M_p$ and
the elementary potential functions $G_p$. We thus have to estimate
$g_{p,n}$ and $b_{p,n}$ with some precision in terms of the $g_p$ and
$b_p$. This task is performed in Section~\ref{estim-generales}. In the
second section, Section~\ref{section-result-unif}, we combine these
estimates with the concentration results presented in Section~\ref{section-Dob-estim} to derive some useful uniform estimates w.r.t. the
time parameter.

\subsection{Semigroup estimates} \label{estim-generales}

We start with a series of technical lemmas.

\begin{lemma} \label{lemmette1}
Let $K$ be a Markov kernel and $G$ a positive function on $E$
satisfying  $ { \sup_{x,y \in E} \frac{G(x)}{G(y)}
\leq g }$, for some finite constant $g$. In this situation, we have that
\[
\sup _{x,y \in E}\frac{K.  G(x)}{K.  G(y)} \leq1 + \beta (K) (g-1).
\]
\end{lemma}

\begin{pf}
Let $x,y \in E$ be s.t. $K.  G(x) \geq K.  G(y)$. Let us write
\[
\frac{K.  G(x)}{K.  G(y)} = \frac{K.  G(x) - K.  G(y)}{K.  G(y)} +1 \leq\frac{\beta(K) (G_{\max} - G_{\min})}{G_{\min}} +1.
\]

We check the last inequality using the fact that
\[
K.  G(x) - K.  G(y) \leq \operatorname{osc}(K.  G) \leq\beta(K). \operatorname{osc} (G) = \beta (K).(G_{\max}-G_{\min})
.
\]
On the other hand, we have $ K.  G(y) = \int_{u}G(u)K(y,\mathrm{d}u) \geq G_{\min} $.
The desired result is now obtained taking the supremum over all $(x,y)
\in E^2$. Note that $\frac{\beta(K) (G_{\max} - G_{\min})}{G_{\min}}
+1$ is exactly equal to $ 1+\beta(K) (g-1)$.

This ends the proof of the lemma.
\end{pf}

\begin{lemma} \label{lemmette2}
Let $M$ be a Markov kernel, $Q$ a not necessarily normalized integral
operator satisfying $ { \sup_{x,y \in E} \frac
{Q .  1(x)}{Q .  1(y)} \leq g }$, for some finite constant $g \geq1$ and $f$
a bounded, nonnegative function. In this situation, the Markov kernel
$P$ defined by
\[
P.f(x):= \frac{M.  Q .  f(x)}{M.  Q .  1(x)}
\]
satisfies the following property:
\[
\beta(P) \leq g. \beta(M). \beta\bigl(P^{\prime}\bigr).
\]

 In the above display formula, $P^{\prime}$ is the Markov
transition defined by $ { P^{\prime}.f(x):= \frac
{Q .  f(x)}{Q .  1(x)} }$.
\end{lemma}

\begin{pf}
Note that $P.f(x)$ can be written in this way $P.f(x) = \Psi
_{Q .  1}(\delta_x.M)  ( P^{\prime}.f  )$. Thus, for any $ x,y
\in E$, we have that
\begin{eqnarray*}
\bigl\llvert P.f(x) - P.f(y) \bigr\rrvert & = & \bigl\llvert \bigl(
\Psi_{Q .  1}(\delta _x.M) - \Psi_{Q .  1}(
\delta_y.M) \bigr) \bigl( P^{\prime}.f \bigr) \bigr\rrvert
\\
& \leq& \bigl\llVert \bigl( \Psi_{Q .  1}(\delta_x.M) -
\Psi_{Q .  1}(\delta _y.M) \bigr) \bigr\rrVert
_{\mathrm{tv}}.\operatorname{osc} \bigl( P^{\prime}.f \bigr).
\end{eqnarray*}

 Lemma~\ref{BG-tv} implies that
\[
\bigl\llVert \bigl( \Psi_{Q .  1}(\delta_x.M) -
\Psi_{Q .  1}(\delta_y.M) \bigr) \bigr\rrVert
_{\mathrm{tv}} \leq g. \llVert \delta_x.M - \delta_y.M
\rrVert _{\mathrm{tv}} \leq g. \beta(M).
\]

 Using (\ref{prop-Dob}), we have $\operatorname{osc}  ( P^{\prime}.f
 ) \leq\beta(P^{\prime}).\operatorname{osc}(f)$.

This ends the proof of the lemma.
\end{pf}

\begin{lemma} \label{estim-g-b}
For any integers $p \leq n$, we have
%
\begin{eqnarray}
g_{p,n}-1 & \leq&\sum_{k=p+1}^n
(g_k -1) \prod_{i=p+1}^{k-1}
(b_{i}g_{i}), \label{estim-g}
\\
b_{p,n} & \leq&\prod_{k=p+1}^n
b_k .  g_{k,n}. \label{estim-b}
\end{eqnarray}
\end{lemma}

\begin{pf}
Let us prove (\ref{estim-g}). By definition, we have $G_{p,n} =
Q_{p,n}.  1$. Combining (\ref{defQ}) and (\ref{defQpn}) applied to unit
function, we have
\[
Q_{p-1,n}(1) = Q_{p}. \bigl[ ( Q_{p+1},\ldots,
Q_n ).  1 \bigr] = G_p \times M_p. (
Q_{p,n}.  1 ).
\]
This implies that the functions $G_{p,n}$ satisfy the following
``backward'' relations:
\[
G_{n,n}=1;\qquad G_{p-1,n}=G_p \times
M_p.G_{p,n}.
\]
Then, for any $x,y\in E$, we deduce that
\[
\frac{G_{p-1,n}(x)}{G_{p-1,n}(y)} = \underbrace{\frac
{G_p(x)}{G_p(y)}}_{E_1} \times
\underbrace{ \frac{  ( M_{p}.G_{p,n}
 )(x) }{ ( M_{p}.G_{p,n}  )(y)} }_{E_2}.
\]

Notice that $E_1 \leq g_p$ (by definition), and by Lemma~\ref
{lemmette1}, we have $E_2 \leq1 + \beta(M_{p}).(g_{p,n}-1)$. This
shows the following backward inequalities:
%
\begin{equation}
\label{estim-rec-g} g_{n,n}=1;\qquad g_{p-1,n} \leq g_p
\bigl( 1+ b_p(g_{p,n}-1) \bigr).
\end{equation}
We end the proof of (\ref{estim-g}) by induction.

To prove (\ref{estim-b}), we use the formulae
\[
P_{p-1,n}.f = \frac{Q_{p-1,n}.f}{Q_{p-1,n}.  1} = \frac{G_p
\times M_p.  Q_{p,n}.f}{G_p \times M_p.  Q_{p,n}.  1} =
\frac{
M_p.  Q_{p,n}.f}{M_p.  Q_{p,n}.  1}.
\]
Recalling that $ { P_{p,n}.f = \frac{Q_{p,n}.f}{Q_{p,n}.  1}
} $, we apply Lemma~\ref{lemmette2} to check that $\beta(P_{p-1,n})
\leq\beta(M_p).g_{p,n} \beta(P_{p,n}) $, from which we conclude that
\[
b_{p-1,n} \leq b_p.g_{p,n}.b_{p,n}.
\]

 We end the proof of (\ref{estim-b}) by induction.

This ends the proof of the lemma.
\end{pf}

 We end this section with a useful technical lemma to control
the quantity $g_{p,n} b_{p,n}$.

\begin{lemma} \label{lemme-gpn-bpn}
For any $p \leq n$, we have
\[
g_{p,n} b_{p,n} \leq\prod_{k=p+1}^n
( b_k.g_{k-1,n} ).
\]
\end{lemma}

\begin{pf}
Using Lemma~\ref{estim-g-b}, we have
\begin{eqnarray*}g_{p,n} b_{p,n} & \leq& {
g_{p,n} .  \prod_{k=p+1}^n
b_k g_{k,n} }
\\
& = & g_{p,n} .  (b_{p+1}g_{p+1,n}) .  (b_{p+2} g_{p+2,n}) \cdots(b_{n-1} g_{n-1,n})
. (b_n \underbrace{g_{n,n}}_{=1} )
\\
& = & (g_{p,n} b_{p+1}) .  (g_{p+1,n}
b_{p+2}) \cdots (g_{n-1,n} b_n) = { \prod
_{k=p+1}^n b_k g_{k-1,n} }.
\end{eqnarray*}
This ends the proof of the lemma.
\end{pf}

The term $g_{p,n} b_{p,n}$ is central in the $L^p$-mean error bound
(\ref{born-Lp-DMM}). By Lemma~\ref{lemme-gpn-bpn}, we have
\[
{\sum_{p=0}^{n} \prod
_{k=p+1}^{n} b_k g_{k-1,n} < +
\infty }\quad\Longrightarrow\quad\sum_{p=0}^{n}
g_{p,n} b_{p,n} < +\infty.
\]
This gives a sufficient condition for a uniform $L^p$ bound w.r.t. time
$n$. $g_{p,n} b_{p,n}$ is also involved in the estimates of all the
quantities defined in Section~\ref{section-Dob-estim} such as $r_n$,
${\overline{\beta_n}}^2$, and others. In addition, by Lemma~\ref
{BG-tv}, we have the stability property
\[
\bigl\Vert\phi_{p,n} (\mu) - \phi_{p,n} (\nu)
\bigr\Vert_{\mathrm{tv}} \leq g_{p,n} b_{p,n} \Vert\mu- \nu
\Vert_{\mathrm{tv}}.
\]

 This shows that the term $g_{p,n} b_{p,n}$ is central to
quantify the stability properties of the semigroup~$\phi_{p,n}$.

\subsection{Uniform concentration theorems}

\label{section-result-unif}

To obtain uniform bounds w.r.t., the time horizon in (\ref
{born-Lp-DMM}), Lemma~\ref{lemme-gpn-bpn} naturally leads to a
sufficient condition of the following type:
\[
\sup_{k \leq n} b_k .  g_{k-1,n} \leq a \qquad\mbox{for
some } a \in(0,1).
\]

 In this situation, we prove that $ { g_{p,n}
b_{p,n} \leq a^{n-p} }$ and, therefore, using (\ref{born-Lp-DMM}),
\[
\sup _{f \in\mathcal{B}_1 (E)}\mathbb{E} \bigl( \bigl\llvert \eta_n^N
(f) - \eta_n (f) \bigr\rrvert ^p \bigr)^{1/p}
\leq\frac{B_p}{\sqrt
{N}} \frac{1}{1-a},
\]
with the constants $B_p$ introduced in (\ref{def-Bp}). We then fix the
parameter $a \in(0,1)$ and we look for conditions on the $b_p$ so that
$b_k g_{k-1,n} \leq a$. This parameter $a$ can be interpreted as a
performance degree of the $N$-approximation model. In order to explicit
relevant and applicable conditions, we study two typical classes of
regularity conditions on the potential functions $G_p$. The first one
relates to bounded coefficients $g_p$ (Theorem~\ref{reg born}). In the
second one, the parameters $g_p$ tend to $1$ as $p \rightarrow\infty$
(Theorem~\ref{reg dec}).


\begin{theorem} \label{reg born}
We assume that
%
\begin{equation}
\label{condition-born} \sup_{p \geq1} g_p \leq M\quad \mbox{and}\quad
\sup_{p \geq1} b_p \leq\frac{a}{M(1+a)}
\end{equation}
for some finite $M \geq1$ and some $a \in(0,1)$. In this
situation, we have the following uniform estimates:
\begin{itemize}
\item The $L^p$-error bound:
%
\begin{equation}
\label{reg-born-Lp} \sup_{n \geq0} d_p \bigl(
\eta_n^N, \eta_n \bigr) \leq
\frac{B_p}{2 (1-a) \sqrt{N} }.
\end{equation}

\item For any $n\geq0$, $N \geq1$, $y \geq0$ and $f \in\mathcal
{O}_1 (E)$, the probability of the event
%
\begin{equation}
\label{reg-born-eta} \eta_n^N(f) - \eta_n(f) \leq
\frac{r_1^{\star}  (1+h_0(y)  )
+ r_2^{\star} \sqrt{Ny} }{N}
\end{equation}
is greater than $1-\mathrm{e}^{-y}$, with the parameters
%
\begin{equation}
\label{def-r-star-12} r_1^{\star}= \frac{4M^2 (1+a)^2 }{1-a} \quad\mbox{and}\quad
r_2^{\star} = \frac{2 \sqrt{2} }{\sqrt{1-a^2}}.
\end{equation}

\item For any $n\geq0$, $N \geq1$, $\epsilon\in\{ +1, -1 \}$,
and $y \geq0$, the probability of the event
%
\begin{equation}
\label{reg-born-gamma} \frac{\epsilon}{n} \log \biggl( \frac{\gamma_n^N (1)}{\gamma_n(1)} \biggr) \leq
\frac{\tilde{r}_1}{N} h_0(y) + \tilde{r}_2.h_1
\biggl( \frac{y}{n.N} \biggr)
\end{equation}
is greater than $1-\mathrm{e}^{-y}$, with the parameters $ \tilde{r}_1
= \frac{8M^3 (1+a)^2 }{1-a} $ and $ \tilde{r}_2 = \frac{4M}{1-a} $. The
functions $h_0$ and $h_1$ are defined in (\ref{def-fonctions-h}),
page \pageref{def-fonctions-h}.
\end{itemize}
\end{theorem}

%
\begin{pf} Firstly, we prove the pair of inequalities
%
\begin{equation}
\label{estim-gpn-prod-born}\cases{ %
 g_{p,n}\leq M (1+a),
\vspace*{2pt}\cr
b_p .  g_{p-1,n}\leq a. }
\end{equation}

 If $g_p \leq M$ and $ {b_p \leq b:= \frac
{a}{M(1+a)}}$, then by Lemma~\ref{estim-g-b}
\[
g_{p,n} - 1 \leq\sum_{k=p+1}^n
(M-1) (bM)^{k-p-1} = (M-1) \frac{1-(bM)^{n-p} }{1-bM} \leq\frac{M-1}{1-bM}.
\]
We substitute $ {b = \frac{a}{M(1+a)}}$ to find
\[
g_{p,n} \leq\frac{M-1}{1-{a}/{(1+a)}} +1 = (M-1) (1+a) +1 \leq M(1+a).
\]
For the second inequality, we have
\begin{eqnarray*}
b_p .  g_{p-1,n} & \leq&\frac{b(M-1)}{1-{a}/{(1+a)}} +b
\\
& =& (1+a) \frac{a(M-1)}{M(1+a)} + \frac{a}{M(1+a)}
\\
& = &\frac{1}{M} \biggl( aM - a + \frac{a}{1+a} \biggr)
\\
& \leq& a.
\end{eqnarray*}
This ends the proof of (\ref{estim-gpn-prod-born}).
\begin{itemize}
\item By Lemma~\ref{lemme-gpn-bpn} and (\ref{estim-gpn-prod-born}), we
have $g_{p,n}b_{p,n} \leq a^{n-p}$. Combining this with (\ref
{born-Lp-DMM}), the $L^p$-error bound (\ref{reg-born-Lp}) is clear. See
the relation (\ref{remarque-osc-born1}), page \pageref
{remarque-osc-born1} for the factor $1/2$.

\item Let us prove (\ref{reg-born-eta}), which is a consequence of the
concentration inequality (\ref{ineg-conc}). Combining the estimations
of $r_n$ and ${\overline{\beta_n}}^2$ given in Section~\ref{section-Dob-estim} with (\ref{estim-gpn-prod-born}) and
$g_{p,n}b_{p,n} \leq a^{n-p}$, we deduce that
\[
r_n \leq\frac{4M^2 (1+a)^2}{1-a} \quad\mbox{and}\quad {\overline {
\beta_n}}^2 \leq\frac{4}{1-a^2}.
\]
Equation~(\ref{reg-born-eta}) is obtained by making the suitable substitutions
in (\ref{ineg-conc}).

\item The last concentration inequality (\ref{reg-born-gamma}) is a
consequence of (\ref{ineg-conc-nor}) and (\ref{estim-gpn-prod-born}).
Let us recall that Lemma~\ref{lemme-gpn-bpn} and (\ref
{estim-gpn-prod-born}) imply $g_{p,n}b_{p,n} \leq a^{n-p}$. Then, from
estimations (\ref{estim-tau}) and (\ref{estim-r}), we can easily show
that the quantities $\tau_n^{\star}$ and $\bar{r}(n)$ satisfy
\[
\tau_n^{\star} \leq\frac{4M}{n(1-a)} \quad\mbox{and}\quad \bar
{r}(n) \leq\frac{8M^3 (1+a)^2 }{1-a}.
\]

 On the other hand, $\bar{\sigma}_n^2$ is trivially bounded by
$n$. Then we find that
\[
\frac{\bar{r} (n) }{N} h_0 (y) + \tau_n^{\star}
\bar{\sigma}_n^2 h_1 \biggl(
\frac{y}{N.\bar{\sigma}_n^2} \biggr) = \frac{\bar{r} (n) }{N} h_0 (y) +
\frac{y \tau_n^{\star}}{3N} + \sqrt{ \frac{2y(\tau_n^{\star}
\bar{\sigma}_n)^2}{N} }.
\]

 Finally, (\ref{reg-born-gamma}) is obtained from (\ref
{ineg-conc-nor}) by making the suitable substitutions.
\end{itemize}
This ends the proof of the theorem.
\end{pf}

Let us now consider the case where $g_p$ decreases to $1$ as $p
\rightarrow\infty$. The idea of the forthcoming analysis is to find a
condition on the $b_p$ so that the $g_{p,n}$ are uniformly bounded
w.r.t. $n$ by $g_{p+1}^{1+\alpha}$ with
\[
\alpha= \frac{a}{1-a} > 0\qquad \biggl( \Longleftrightarrow a =
\frac
{\alpha}{1 + \alpha} \biggr).
\]

 The concentration inequalities developed in Theorem~\ref{reg
dec} will be described in terms of the parameters $r_3^{\star}(n)$ and
$\tilde{r}_3(n),\tilde{r}_4,\tilde{r}_5(n)$ defined below.
%
\begin{eqnarray}
\label{def-r-star-34} &&r_3^{\star} (n) = \frac{4 u_1 (n)}{1-a},
\\
\label{def-r-tilde-345}&& \cases{ %
\displaystyle\tilde{r}_3(n)=
{ \frac{16.u_2(n)}{1-a} },
\vspace*{2pt}\cr
\displaystyle\tilde{r}_4 = { \frac{4}{3} \sum
_{n \geq0} g_{n+1}.a^n }  \leq {
\frac{4.g_1}{3(1-a)} },
\vspace*{2pt}\cr
\displaystyle\tilde{r}_5(n) = { \frac{4\sqrt{2}. u_3(n) }{1-a} }.}
\end{eqnarray}

 The sequences $u_1(n)$, $u_2(n)$ and $u_3(n)$ used in the
above formulae are defined by
\[
\cases{ \displaystyle u_1(n) = (1-a) \sum
_{p = 0}^n g_{n-p+1}^{2(1+\alpha)}
a^p \mathop{\longrightarrow}_{n \rightarrow\infty} 1,
\vspace*{2pt}\cr
\displaystyle u_2(n) = \frac{1}{n} \sum_{p=1}^n
g_p^{3+2\alpha}\mathop {\longrightarrow}_{n \rightarrow\infty} 1,
\vspace*{2pt}\cr
\displaystyle u_3(n) = \Biggl( \frac{1}{n} \sum
_{p=0}^{n-1} g_{p+1}^2
\Biggr)^{1/2} \mathop{\longrightarrow} _{n \rightarrow\infty}1. }
\]
Notice that the sequence $u_1(n)$ tends to $1$ by dominated
convergence. Sequences $u_2(n)$ and $u_3(n)$ tend to $1$ by Cesaro's theorem.\


\begin{theorem} \label{reg dec}

We assume that $g_p \downarrow1$ as $p \rightarrow\infty$ and the
sequence $b_p$ satisfies for any $p \geq1$
%
\begin{equation}
\label{cond-bp-gpun} b_p \leq{\frac{g_p^{\alpha} -1} {g_p^{\alpha+1} -1}} ( \mathop{\longrightarrow}_{p \rightarrow+\infty}
a ) \quad\mbox{and}\quad b_p \leq{\frac{a} {g_p^{\alpha+1}}} ( \mathop{
\longrightarrow}_{p
\rightarrow+\infty} a ).
\end{equation}

 In this situation, we have the following uniform estimates:
\begin{itemize}
\item The $L^p$-error bound
%
\begin{equation}
\label{reg-dec-Lp} \sup_{n \geq0} d_p \bigl(
\eta_n^N, \eta_n \bigr) \leq
\frac{B_p}{2 (1-a) \sqrt{N} }
\end{equation}
with the constants $B_p$ introduced in (\ref{def-Bp}).

\item For any $n\geq0$, $N \geq1$, $y \geq0$ and $f \in\mathcal
{O}_1 (E)$, the probability of the event
%
\begin{equation}
\label{reg-dec-eta} \eta_n^N(f) - \eta_n(f) \leq
\frac{r_3^{\star} (n)  (1+h_0(y)
 ) + r_2^{\star} \sqrt{Ny} }{N}
\end{equation}
is greater than $1-\mathrm{e}^{-y}$, with the bounded sequence
$r_3^{\star} (n)$ defined in (\ref{def-r-star-34}), the parameter
$r_2^{\star}$ defined in (\ref{def-r-star-12}) and the function $h_0$
defined in (\ref{def-fonctions-h}), page \pageref{def-fonctions-h}.

\item For any $n\geq0$, $N \geq1$, $\epsilon\in\{ +1, -1 \}$,
and $y \geq0$, the probability of the event
%
\begin{equation}
\label{reg-dec-gamma} \frac{\epsilon}{n} \log \biggl( \frac{\gamma_n^N (1)}{\gamma_n(1)} \biggr) \leq
\tilde{r}_3(n) \biggl( \frac{y + \sqrt{y} }{N} \biggr) +
\tilde{r}_4 \biggl( \frac{y}{n.N} \biggr) +
\tilde{r}_5(n) \sqrt{ \frac{y}{n.N} }
\end{equation}
is greater than $1-\mathrm{e}^{-y}$, with the parameters $\tilde
{r}_3(n),\tilde{r}_4,\tilde{r}_5(n)$ defined in (\ref{def-r-tilde-345}).
\end{itemize}
\end{theorem}

%
\begin{pf}
Firstly, we prove that
%
\begin{equation}
\label{estim-gpn-prod-dec} \mbox{(\ref{cond-bp-gpun})}\quad
 \Longrightarrow\quad \forall p \leq n,\qquad \cases{
g_{p,n} \leq (g_{p+1})^{1 + \alpha},
\vspace*{2pt}\cr
g_{p-1,n} .  b_p  \leq a. }
\end{equation}

 The proof of the first inequality comes from a simple
backward induction on $p$ (with fixed $n$), using formula $g_{p-1,n}
\leq g_p  ( 1+ b_p(g_{p,n}-1)  )$ (see (\ref{estim-rec-g})).
For $p=n$, $g_{p,n}$ is clearly smaller than $g_{p+1}^{1+\alpha}$
because $g_{n,n} = 1$. The second assertion is now immediate.

 Now we assume that $g_{p,n} \leq g_{p+1}^{1+\alpha}$. In this
case, $g_{p-1,n} \leq g_{p}^{1+\alpha}$ is met as soon as
\[
b_p \leq\frac{g_p^{\alpha} -1} {g_{p+1}^{\alpha+1} -1}.
\]

 Notice that this estimate is met as soon as $b_p \leq\frac
{g_p^{\alpha} -1}{g_{p}^{\alpha+1} -1}$, the sequence $(g_p)_p$ being
decreasing.
\begin{itemize}
\item Now that we proved (\ref{estim-gpn-prod-dec}) (which implies
$g_{p,n}b_{p,n} \leq a^{n-p}$ by Lemma~\ref{lemme-gpn-bpn}), the
$L^p$-mean error bound (\ref{reg-dec-Lp}) comes from a simple
substitution in (\ref{born-Lp-DMM}). See the relation (\ref
{remarque-osc-born1}), page~\pageref{remarque-osc-born1} for the factor $1/2$.
\item To prove (\ref{reg-dec-eta}), we focus on the quantities
${\overline{\beta_n}}$ and $r_n$ arising in the concentration
inequality (\ref{ineg-conc}). With (\ref{estim-gpn-prod-dec}) and
$g_{p,n}b_{p,n} \leq a^{n-p}$, we readily verify that
\[
{\overline{\beta_n}}^2 \leq\frac{4}{1-a^2}.
\]

 The term $r_n$ can be roughly bounded by $\frac{4}{1-a}
g_1^{2(1+\alpha)}$, but another manipulation provides a more precise
estimate. Indeed, using the fact that $b_{p,n} .  g_{p,n} \leq a^{n-p}$
and $g_{p,n} \leq g_{p+1}^{1+ \alpha}$, we prove that
\begin{eqnarray*}r_n & \leq& {4 \sum
_{p=0}^{n} g_{p,n}^2
a^{n-p} }  \leq { 4 \sum_{p=0}^{n}
g_{p+1}^{2(1+
\alpha)} a^{n-p} } \leq { 4 \sum
_{p=0}^{n} g_{n-p+1}^{2(1+
\alpha)}
a^{p} }
\\
&\leq& { \frac{4.u_1(n)}{1-a} }.
\end{eqnarray*}

We prove (\ref{reg-dec-eta}) by making the suitable
substitutions in (\ref{ineg-conc}).

\item Let us prove the last concentration inequality (\ref
{reg-dec-gamma}). It is mainly a consequence of the inequality (\ref
{ineg-conc-nor}). Starting from the following decomposition,
%
\begin{equation}
\label{osef} \frac{\bar{r} (n) }{N} h_0 (y) + \tau_n^{\star}
\bar{\sigma}_n^2 h_1 \biggl(
\frac{y}{N.\bar{\sigma}_n^2} \biggr) = \frac{2 \bar{r} (n) }{N} ( y + \sqrt{y} ) +
\frac{y \tau_n^{\star}}{3N} + \sqrt{ \frac
{2y(\tau_n^{\star} \bar{\sigma}_n)^2}{N} },
\end{equation}
we need to find some refined estimates of the quantities $\tau
_n^{\star}$, $\bar{r}_n$ and $(\tau_n^{\star} .  \bar{\sigma}_n)^2$. To
estimate $\tau_n^{\star}$, we notice that $\forall q, g_{p+1} \leq
g_{p-q+1}$, so that
\begin{eqnarray*}\tau_{q,n} & \leq& {
\frac{4}{n} \sum_{p=q}^{n-1}
g_{p+1} a^{p-q} }  \leq { \frac{4}{n} \sum
_{p=q}^{n-1} g_{p-q+1} a^{p-q} }
\leq { \frac{4}{n} \sum_{p=0}^{n-q-1}
g_{p+1} a^p }
\\
&\leq& { \frac{4}{n} \sum_{p=0}^{n-1}
g_{p+1} a^p } .
\end{eqnarray*}
Finally, we have that $ { \tau_n^{\star} \leq\frac
{4.U_0}{n} }$, where $ { U_0 = \sum_{p \geq0} g_{p+1} a^p
\leq\frac{g_1}{1-a} }$.

We estimate $\bar{r}_n$, using the following inequalities:
\begin{eqnarray*} \bar{r}_n & \leq& {
\frac{8}{n} \sum_{0\leq
q \leq p <n} g_{p+1} .  g_{q,p}^3 .  b_{q,p} }  \leq {
\frac{8}{n} \sum_{p=0}^{n-1} \sum
_{q=0}^p g_{p+1}^{3+2\alpha}
a^{p-q} }
\\
& \leq& { \frac{8}{1-a}  \cdot \underbrace {\frac{1}{n} \sum
_{p=1}^{n} g_{p}^{3+2\alpha}
}_{=u_2(n)} }.
\end{eqnarray*}
Let us conduct a last useful estimation:
\begin{eqnarray*} \bigl(\tau_n^{\star} .  \bar{
\sigma}_n\bigr)^2 & \leq& { \sum
_{q =0}^{n-1} \tau_{q,n}^2}
\leq {\sum_{q=0}^{n-1} \Biggl(
\frac{4}{n} \sum_{p=q}^{n-1}
g_{p+1} .  a^{p-q} \Biggr)^2 }
\\
&\leq& { \sum_{q =0}^{n-1} \Biggl(
\frac
{4}{n} \sum_{p=0}^{n-q+1}
g_{p+q+1} .  a^p \Biggr)^2 }
\\
& \leq& { \sum_{q =0}^{n-1}
\frac
{16}{n^2}  \cdot  g_{q+1}^2  \cdot \frac{1}{(1-a)^2} }
\\
&\leq& { \frac{16}{n.(1-a)^2} \times \underbrace{ \frac{1}{n} \sum
_{q =0}^{n-1} g_{q+1}^2}_{=
(u_3(n) )^2}
}.
\end{eqnarray*}
At last, we make the suitable substitutions in (\ref{osef}) and obtain
the desired inequality (\ref{reg-dec-gamma}).
\end{itemize}
This ends the proof of the theorem.
\end{pf}


\section{Interacting simulated annealing models}
\label{section-Gibbs}

\subsection{Some motivations} \label{section-Gibbs-motiv}

We consider the Boltzmann--Gibbs probability measure associated with an
inverse ``temperature'' parameter $\beta\geq0$ and a given potential
function $V \in\mathcal{B}(E)$ defined by
%
\begin{equation}
\label{def-mes-Gibbs} \mu_{\beta} (\mathrm{d}x) = \frac{1}{Z_{\beta}} \mathrm{e}^{- \beta.V(x)}
m(\mathrm{d}x),
\end{equation}
where $m$ stands for some reference measure, and $Z_{\beta}$ is a
normalizing constant. We let $\beta_n$ a strictly increasing sequence
(which may tend to infinity as $n \rightarrow\infty$). In this case,
the measures $\eta_n = \mu_{\beta_n}$ can be interpreted as a FK flow
of measures with potential functions $G_n = \mathrm{e}^{-(\beta_n - \beta
_{n-1})V}$ and Markov transitions $M_n$ chosen as being MCMC dynamics
for the current target distributions. Indeed, we have
\[
\mu_{\beta_{n}} (\mathrm{d}x) = \frac{\mathrm{e}^{- \beta_n.V(x)}}{Z_{\beta_n}} m(\mathrm{d}x) = \frac{ Z_{\beta_{n-1}} } {Z_{\beta_n}}
\underbrace{\mathrm{e}^{-(\beta_n - \beta
_{n-1})V(x)}}_{=G_n(x)} \underbrace{ \biggl(
\frac{\mathrm{e}^{- \beta
_{n-1}.V(x)}}{Z_{\beta_{n-1}}} m(\mathrm{d}x) \biggr) }_{=\mu_{\beta_{n-1}} (\mathrm{d}x)}.
\]
This shows $ { \mu_{\beta_{n}} = \psi_{G_n} (\mu_{\beta
_{n-1}}) } $. Let $\phi_n$ stand for the FK transformation associated
with potential function $G_n$ and Markov transition $M_n$. We have
\[
\phi_n(\mu_{\beta_{n-1}}) = \psi_{G_n} (
\mu_{\beta_{n-1}}) = \mu_{\beta
_{n}} .  M_n = \mu_{\beta_{n}}
.
\]

Sampling from these distributions is a challenging problem in many
application domains. The simplest one is to sample from a complex
posterior distribution on some Euclidian space $E=\mathbb{R}^d$, for
some $d \geq1$. Let $x$ be a variable of interest, associated with a
prior density $p(x)$ (easy to sample) with respect to Lebesgue measure
$\mathrm{d}x$ on $E$, and $y$ a vector of observations, associated with a
calculable likelihood model $p(y \mid x)$. In this context, we recall
that $p(y \mid x)$ is the density of the observations given the
variable of interest. The density $p(x \mid y)$ of the posterior
distribution $\eta$ is given by Bayes' formula
\[
p(x \mid y) \propto p(x). p(y \mid x).
\]
In the case where $\eta$ is highly multimodal, it is difficult to
sample from it directly. As an example, classic MCMC methods tend to
get stuck in local modes for very long times. As a result, they
converge to their equilibrium $\eta$ only on unpractical time-scales.
To overcome this problem, a common solution is to approximate the
target distribution $\eta$ with a sequence of measures $\eta_0, \ldots
, \eta_n$ with density
\[
\eta_k(\mathrm{d}x) \propto p(x). p(y \mid x)^{\beta_k} \,\mathrm{d}x,
\]
where $ { (\beta_k)_{0 \leq k \leq n} }$ is a sequence of
number increasing from $0$ to $1$, so that $\eta_0$ is the prior
distribution of density $p(x)$, easy to sample, and the terminal
measure $\eta_{n}$ is the target distribution~$\eta$ (see, for
instance, \cite{Bertrand,Giraud-RB,Minvielle,Neal}). If we take $V:= x
\mapsto- \log(p(y \mid x))$ and $m(\mathrm{d}x):= p(x) \,\mathrm{d}x$, then the $\eta_k$
coincide with the Boltzmann--Gibbs measures $\mu_{\beta_k}$ defined in
(\ref{def-mes-Gibbs}). In this context, IPS methods arise as being a
relevant approach, especially if $\eta$ is multimodal, since the use of
a large number of particles allows to cover several modes
simultaneously.

The normalizing constant $\gamma_n(1)$ coincide with the marginal
likelihood $p(y)$. Computing this constant is another central problem
in model selection arising in hidden Markov chain problems and Bayesian
statistics.

Next, we present another important application in physics and
chemistry, known as free energy estimation. The problem starts with an
unnormalized density of the form
\[
q ( \omega\mid T,\alpha ) = \exp \biggl( -\frac{H(\omega,
\alpha)}{k.T} \biggr),
\]
where $H(\omega, \alpha)$ is the energy function of state $\omega$, $k$
is Boltzmann's constant, $T$ is the temperature and $\alpha$ is a vector
of system characteristics. The free energy $F$ of the system is defined
by the quantity
\[
F(T, \alpha) = -k.T.\log\bigl(z(T, \alpha)\bigr),
\]
where $z(T, \alpha)$ is the normalizing constant of the system density.
See, for instance, \cite{Ceperley,Ciccotti-Hoover,Frankel-Smit} for a
further discussion on these ground state energy estimation problems.

Last, but not least, it is well known that Boltzmann--Gibbs measures'
sampling is related to the problem of minimizing the potential function
$V$. The central idea is that $\mu_{\beta}$ tends to concentrate on
$V$'s minimizers as the inverse temperature $\beta$ tends to infinity.
To be more precise, we provide an exponential concentration inequality
in Lemma~\ref{lemme-conc-Gibbs}.

In this context, the IPS algorithm can be interpreted as a sequence of
interacting simulated annealing (\textit{abbreviate ISA}) algorithms. As
they involve a population of $N$ individuals, evolving according to
genetic type processes (selection, mutation), ISA methods also belong
to the rather huge class of evolutionary algorithms for global
optimization. These algorithms consist in exploring a state space with
a population associated with an evolution strategy, that is, an
evolution based on selection, mutation and crossover. See \cite
{Goldberg} or \cite{Le-Riche} and the references therein for an
overview. As these algorithms involve complex, possibly adaptive
strategies, their analysis is essentially heuristic, or sometimes
asymptotic (see \cite{Cerf} for general convergence results on genetic
algorithms). The reader will also find in \cite{DM-M-Efini} a proof for
a ISA method of the a.s. convergence to the global minimum in the case
of a finite state space, when the time $n$ tends to infinity, and as
soon as the population size $N$ is larger than a critical constant that
depends on the oscillations of the potential fitness functions and the
mixing properties of the mutation transitions.

The results of the previous sections apply to the analysis of ISA
optimization methods. Our approach is nonasymptotic since we estimate
at each time $n$, and for a fixed population size $N$ the distance
between the theoretical Boltzmann--Gibbs measure $\eta_n$ and its
empirical approximation $\eta_n^N$. In some sense, as $\beta_n
\rightarrow+ \infty$, $\eta_n$ tends to the Dirac measure $\delta
_{x^{\star}}$ where $x^{\star} = {\operatorname{argmin}}_{x \in E}
V(x)$, as soon as there is a single global minimum. So intuitively, we
guess that if this distance admits a uniform bound w.r.t. time $n$,
then for large time horizon we have $\eta_n^N \simeq\delta_{x^{\star}}
$.

In this section, we propose to turn the conditions of Theorems \ref{reg
born} and \ref{reg dec} into conditions on the temperature
schedule to use, and the number of MCMC steps that ensure a given
performance degree. Then we combine the concentration results of
Section~\ref{section-result-unif} with Lemma~\ref{lemme-conc-Gibbs} to
analyze the convergence of the IPS optimization algorithm.

\subsection{An ISA optimization model}

\label{section-Gibbs-tuning}

We fix an inverse temperature schedule $\beta_n$ and we set
\begin{itemize}

\item$\eta_n(\mathrm{d}x) = \mu_{\beta_n}(\mathrm{d}x)= \frac{1}{Z_{\beta_n}} \mathrm{e}^{- \beta
_n V(x)} m(\mathrm{d}x)$;
\item$G_n(x) = \mathrm{e}^{-\Delta_n.V(x)}$;
\item and then $g_n = \mathrm{e}^{\Delta_n .  \operatorname{osc} (V)}$,
\end{itemize}
where $\Delta_n$ are the increments of temperature $\Delta_n
= \beta_n - \beta_{n-1}$. We let $K_{\beta}$ the simulated annealing
Markov transition with invariant measure $\mu_{\beta}$ and a
proposition kernel $K(x,\mathrm{d}y)$ reversible w.r.t. $m(\mathrm{d}x)$. We recall that
$K_{\beta}(x,\mathrm{d}y)$ is given by the following formulae:
\begin{eqnarray*}
 K_{\beta}(x,\mathrm{d}y)&=& K(x,\mathrm{d}y). \min \bigl( 1
, \mathrm{e}^{-\beta ( V(y) - V(x)
 )} \bigr) \qquad \forall y \neq x,
\\
K_{\beta}\bigl(x,\{ x \}\bigr) &=& 1 - \int_{y \neq x}
K(x,\mathrm{d}y). \min \bigl( 1, \mathrm{e}^{-\beta ( V(y) - V(x)  )} \bigr).
\end{eqnarray*}
Under the assumption $K^{k_0}(x, \cdot ) \geq\delta\nu( \cdot )$ for any
$x$ with some integer $k_0 \geq1$, some measure $\nu$ and some $\delta
>0$, one can show (see, e.g., \cite{Bartoli}, Lemma~4.3.11) that
%
\begin{equation}
\beta\bigl(K_{\beta}^{k_0}\bigr) \leq \bigl( 1- \delta
\mathrm{e}^{- \beta\overline
{\Delta V}(k_0)} \bigr). \label{dob-MH}
\end{equation}

 In the above display formula, $ {\overline{\Delta
V}(k_0) = \sup_{x \in E} \llVert  V(x)-V(X(x)) \rrVert _{\infty
} }$, where for all $x\in E$, $X(x)$ is a random variable of law $\delta
_x .  K^{k_0}$. In other words, $\overline{\Delta V}(k_0)$ is the
maximum potential gap one can obtain making $k_0$ elementary moves with
the Markov transition $K$. It is clearly bounded by $\operatorname{osc}(V)$. One way
to control the mixing properties of the ISA model is to consider the
Markov transition $M_p = K_{\beta_p}^{{k_0}.{m_p}}$, the simulated
annealing kernel iterated $k_0.m_p$ times. In this case, the user has a
choice to make on two tuning parameters, namely the temperature
schedule $\beta_p$, and the iteration numbers $m_p$. Note that for all
$b \in(0,1)$, condition $b_p \leq b$ in turned into $ ( 1 - \delta
\mathrm{e}^{- \beta_p \overline{\Delta V}(k_0)}  )^{m_p} \leq b$, and this
last condition is satisfied if we have
%
\begin{equation}
\label{reg-mp} m_p \geq\frac{\log({1}/{b})\mathrm{e}^{\overline{\Delta V}(k_0). \beta
_p}}{\delta}.
\end{equation}
 We prove now a technical lemma that we will use in the
following. It deals with Boltzmann--Gibbs measures' concentration properties.

\begin{lemma} \label{lemme-conc-Gibbs}
For any $\beta> 0$, and for all $0 < \varepsilon^{\prime} < \varepsilon
$, the Boltzmann--Gibbs measure $\mu_{\beta}$ satisfies
\[
\mu_{\beta} ( V \geq V_{\min} + \varepsilon ) \leq
\frac
{\mathrm{e}^{-\beta(\varepsilon- \varepsilon^{\prime}) }}{m_{\varepsilon^{\prime
}}},
\]
where $m_{\varepsilon^{\prime}} = m  ( V \leq V_{\min} +
\varepsilon^{\prime}  ) >0$.
\end{lemma}

\begin{pf}
The normalizing constant $Z_{\beta}$ of the definition (\ref
{def-mes-Gibbs}) is necessary equal to ${ {\int}}_{E} \mathrm{e}^{-\beta V} \,\mathrm{d}m $. Then we have
\begin{eqnarray*}
\mu_{\beta} ( V \geq V_{\min} + \varepsilon ) & =&
\frac{ { {\int}}_{V \geq V_{\min} + \varepsilon}
\mathrm{e}^{-\beta V} \,\mathrm{d}m }{ { {\int}}_{V \geq V_{\min} + \varepsilon
} \mathrm{e}^{-\beta V} \,\mathrm{d}m + { {\int}}_{V < V_{\min} +
\varepsilon} \mathrm{e}^{-\beta V} \,\mathrm{d}m }
\\
& \leq&\underbrace{ \biggl( { {\int}} _{V \geq V_{\min} +
\varepsilon}\mathrm{e}^{-\beta V} \,\mathrm{d}m \biggr)
}_{A_1} \underbrace{ \biggl( {{\int}}_{V < V_{\min} + \varepsilon} \mathrm{e}^{-\beta V}
\,\mathrm{d}m \biggr)^{-1} }_{A^{-1}_2}.
\end{eqnarray*}
Firstly, it is clear that $ { A_1 \leq \mathrm{e}^{- \beta(V_{\min}
+ \varepsilon)} }$. Secondly, $\varepsilon^{\prime} < \varepsilon$
implies $\{ V \leq V_{\min} + \varepsilon^{\prime} \} \subset\{ V <
V_{\min} + \varepsilon\} $, then we have
\[
A_2 \geq{ {\int}}_{V \leq V_{\min} + \varepsilon^{\prime
}} \mathrm{e}^{-\beta V} \,\mathrm{d}m \geq m \bigl(
V \leq V_{\min} + \varepsilon^{\prime} \bigr) \mathrm{e}^{- \beta(V_{\min} +
\varepsilon^{\prime} )}.
\]
We end the proof by making the appropriate substitutions.
\end{pf}

Combining Lemma~\ref{lemme-conc-Gibbs}, the theorems of Section~\ref{section-result-unif} (with indicator test function $f = \mathbf
{1}_{\lbrace V \geq V_{\min} + \varepsilon\rbrace}$), and the
Dobrushin ergodic coefficient estimate (\ref{dob-MH}) (and the
associated remark~(\ref{reg-mp})) we prove the following theorem, which
is the complete version of Theorem~\ref{theo-statement-optim}.

\begin{theorem} \label{theo-optim-base}
Let us consider the N-particles ISA algorithm associated with the
potential functions $G_n$ and Markov transitions $M_n$ defined in the
beginning of this section. Let us fix $a \in(0,1)$. For any
$\varepsilon>0$, and $n\geq0$, let $p_n^N (\varepsilon)$ denote the
proportion of particles $(\zeta_n^i)$ s.t. $ V (\zeta_n^i) \geq V_{\min
} + \varepsilon$. We assume that the inverse temperature schedule $\beta
_p$ and the iteration numbers $m_p$ satisfy one of these two conditions:
\begin{longlist}[1.]
\item[1.]$ \sup_{p \geq1} \Delta_p \leq\Delta< \infty$
(e.g., linear temperature schedule) and
\[
{m_p \geq\frac{\log({\mathrm{e}^{\Delta.  \operatorname
{osc}(V)}(1+a)}/{a})\mathrm{e}^{\overline{\Delta V}(k_0). \beta_p}}{\delta}}.
\]
\item[2.]$\Delta_p \downarrow0$ (as $p \rightarrow\infty$) and
$ { m_p \geq ( \operatorname{osc}(V). \Delta_p + \log(\frac
{1}{a})  ) \frac{\mathrm{e}^{\overline{\Delta V}(k_0). \beta_p}}{\delta}}$.
\end{longlist}

 In this situation, for any $\varepsilon>0$, $n\geq0$, $N
\geq1$, $y \geq0$, and $\varepsilon^{\prime} \in(0, \varepsilon)$,
the probability of the event
\[
p_n^N (\varepsilon) \leq\frac{\mathrm{e}^{-\beta_n(\varepsilon-
\varepsilon^{\prime}) }}{m_{\varepsilon^{\prime}}} +
\frac{r_i^{\star}
(n)  (1+h_0(y)  ) + r_2^{\star} \sqrt{Ny} }{N}
\]
is greater than $1-\mathrm{e}^{-y}$. In the above display formula, the
constant $r_2^{\star}$ is defined in (\ref{def-r-star-12}), page~\pageref
{def-r-star-12}, the function $h_0$ in (\ref{def-fonctions-h}),
page~\pageref{def-fonctions-h} and $r_i^{\star} (n)$ is defined by:
\begin{itemize}[--]
\item[--] $r_i^{\star} (n) = r_1^{\star}$ (see (\ref{def-r-star-12}),
page \pageref{def-r-star-12} with $M=\mathrm{e}^{\Delta.  \operatorname{osc}(V)}$) in the case of
bounded $\Delta_{p}$;
\item[--] $r_i^{\star} (n) = r_3^{\star} (n)$ (see (\ref
{def-r-star-34}), page \pageref{def-r-star-34} with $g_n=\mathrm{e}^{\Delta_n
\operatorname{osc}(V)}$) in the second one.
\end{itemize}
\end{theorem}

We distinguish two error terms. The first one, $ {  (
\frac{\mathrm{e}^{-\beta_n(\varepsilon- \varepsilon^{\prime}) }}{m_{\varepsilon
^{\prime}}}  ) }$, is related to the concentration of the
Boltzmann--Gibbs measure around the set of global minima of $V$. The
second one, $ {  ( \frac{r_i^{\star} (n)  (1+h_0(y)
 ) + r_2^{\star} \sqrt{Ny} }{N}  ) }$, is related to the
concentration of the occupation measure around the limiting
Boltzmann--Gibbs measure. Besides the fact that Theorem~\ref
{theo-optim-base} provides tuning strategies which ensure the
performance of the ISA model, the last concentration inequality is
explicit in the relative importance of other parameters, including the
probabilistic precision $y$, the threshold $t$ on the proportion of
particles possibly out of the area of interest, the final temperature
$\beta_n$ and the population size $N$. As long as one can have rough
estimates of $m_{\varepsilon^{\prime}}$ and $\operatorname{osc}(V)$ (spatial scale and
variation of the function $V$), a simple equation, deduced from this
last theorem, such as $ {  ( \frac{\mathrm{e}^{-\beta
_n(\varepsilon- \varepsilon^{\prime}) }}{m_{\varepsilon^{\prime}}} =
\frac{r_i^{\star} (n)  (1+h_0(y)  ) + r_2^{\star} \sqrt{Ny}
}{N} = \frac{t}{2}  ) } $ provides some orders of magnitude that
can be useful to roughly design an ISA model at first sight, which is
generally a difficult task.

One natural way to choose $\Delta_p$ is to look at $n_1$, the number of
transitions of type $K_{\beta}^{k_0}$ we need to proceed to move from
$\beta_p$ to $\beta_n$, and to compare it to $n_2$, the number of
transitions we need to proceed to move from $\beta_p$ to $\beta_q$, and
then from $\beta_q$ to $\beta_n$, with $\beta_p < \beta_q < \beta_n$.
Roughly speaking, we have seen that the convergence condition was $b_p
\leq\frac{a}{g_p}$, then, using (\ref{reg-mp}), we have
\begin{itemize}
\item$n_1 \simeq ( \operatorname{osc}(V).\Delta_{p,n}+ \log(\frac{1}{a})
) \frac{\mathrm{e}^{\overline{\Delta V}(k_0). \beta_n}}{\delta}$,
\item$n_2 \simeq ( \operatorname{osc}(V).\Delta_{p,q}+ \log(\frac{1}{a})
) \frac{\mathrm{e}^{\overline{\Delta V}(k_0). \beta_q}}{\delta} +
 ( \operatorname{osc}(V). \Delta_{q,n}+ \log(\frac{1}{a})  ) \frac
{\mathrm{e}^{\overline{\Delta V}(k_0). \beta_n}}{\delta}$,
\end{itemize}
where $\Delta_{p,q}:= \beta_p - \beta_q$ for $p > q$. After some
approximation technique we find that
\[
n_1 \leq n_2 \quad\Longleftrightarrow\quad \Delta_{p,q}
\Delta_{q,n} \geq\frac
{\log({1}/{a} )}{\operatorname{osc}(V). \overline{\Delta V}(k_0)}.
\]

This condition does not bring any relevant information in the
case where $\Delta_p \longrightarrow0$ except that the error
decomposition $\eta_n^N-\eta_n = \sum_{p=0}^{n} \phi_{p,n}(\eta_p^N)-
\phi_{p,n} \phi_{p}(\eta_{p-1}^N)$, underlying our analysis, is not
adapted to the case where $\Delta_p \longrightarrow0$ (which can be
compared to the continuous time case). Nevertheless, this condition is
interesting in the case of constant inverse temperature steps. In this
situation, the critical parameter $\Delta_{\beta}$ is given by
\[
\Delta_{\beta} = \sqrt{ \frac{\log({1}/{a} )}{\operatorname{osc}(V).\overline
{\Delta V}(k_0)} }.
\]

 More precisely, above $\Delta_{\beta}$ the algorithm needs to
run too many MCMC steps to stabilize the system. In the reverse angle,
when the variation of temperature is too small, it is difficult to
reach the desired target measure.


\section{An adaptive temperature schedule in ISA}
\label{section-adapt}

As we already mentioned in the \hyperref[sec1]{Introduction}, the theoretical tuning
strategies developed in Section~\ref{section-Gibbs-tuning} are of the
same order as the logarithmic cooling schedule of traditional simulated
annealing \cite{Bartoli,Cerf,DM-M-Efini}. In contrast to SA models,
we emphasize that the performance of the ISA models are not based on a
critical initial temperature parameter. Another advantage of the ISA
algorithm is to provide at any time step an $N$-approximation of the
target measure with a given temperature. In other words, the population
distribution reflects the probability mass distribution of the
Boltzmann--Gibbs measure at that time. Computationally speaking, the
change of temperature parameter $\Delta_p$ plays an important role. For
instance, if $\Delta_p$ is taken too large, the selection process is
dominated by a minority of well fitted particles and the vast majority
of the particles are killed. The particle set's diversity, which is one
of the main advantage of the ISA method, is then lost. On the contrary,
if $\Delta_p$ is taken too small, the algorithm does not proceed to an
appropriate selection. It wastes time by sampling from MCMC dynamics
while the set of particles has already reached its equilibrium. The
crucial point is to find a relevant balance between maintaining
diversity and avoiding useless MCMC operations.

Designing such a balance in advance is almost as hard as knowing the
function $V$ in advance. Therefore, it is natural to implement adaptive
strategies that depend on the variability and the adaptation of the
population particles (see, e.g., \cite{Jasra,Schafer,Clapp,Deutscher,Minvielle} for related applications). In the general
field of evolutionary algorithms, elaborating adaptive selection
strategies is a crucial question (see, e.g., \cite{Baker}) and a
challenging problem to design performant algorithms. In the case of ISA
methods, the common ways to choose $\Delta_p$ are based on simple
criteria such as the expected number of particle killed (see
Section~\ref{description-algo-adapt}), or the variance of the weights
(effective sample size). All of these criteria are based on the same
intuitive idea; that is to achieve a reasonable selection. As a result,
all of these adaptive ISA models tend to perform similarly.

In \cite{DM-D-J-adapt}, the reader will find a general formalization of
adaptive IPS algorithms. The idea is to define the adaptation as the
choice of the times $n$ at which the resampling occurs. These times are
chosen according to some adaptive criteria, depending on the current
particle set, or more generally on the past process. Under weak
conditions on the criteria, it is shown how the adaptive process
asymptotically converges to a static process involving deterministic
interaction times when the population size tends to infinity. A
functional central limit theorem is then obtained for a large class of
adaptive IPS algorithms.

The approach developed in the following section is radically different,
with a special focus on nonasymptotic convergence results for the ISA
algorithm defined in Section~\ref{description-algo-adapt}. The
adaptation consists here in choosing the $\beta$ increment ${\Delta
}_{n+1}^N$ so that
\[
\eta_n^N \bigl(\mathrm{e}^{-{\Delta}_{n+1}^N  \cdot  V}\bigr) = \varepsilon,
\]
where $\varepsilon> 0$ is a given constant, at each iteration $n$. We
show that the associated stochastic process can be interpretated as a
perturbation of the limiting FK flow.

\subsection{Feynman--Kac representation} \label{FK-representation-adapt}

Let $V \in\mathcal{B}(E)$. To simplify the analysis, without any loss
of generality, we assume $V_{\min}=0$. Let us fix $\varepsilon> 0$.
For any measure $\mu\in\mathcal{P}(E)$, we let the function $\lambda
_{\mu}$ defined by
\begin{eqnarray*}&& [0,+\infty)  \longrightarrow (0,1],
\\
&&\lambda_{\mu} =  x  \mapsto \mu \bigl( \mathrm{e}^{-x  \cdot  V} \bigr).
\end{eqnarray*}

This function is clearly decreasing ($\lambda_{\mu} (0) =
1$), convex, and differentiable infinitely. Moreover, if $\mu ( \{
V=0 \}  ) =0$, then it satisfies $\lambda_{\mu}(x) \longrightarrow
0$ when $x \to+\infty$. Therefore, we can define its inverse function
$\kappa_{\mu}$:
\begin{eqnarray*}&& (0,1]  \longrightarrow [0,+\infty),
\\
&&\kappa_{\mu} =  \varepsilon \mapsto x\qquad \mbox{so that } \mu \bigl(
\mathrm{e}^{-x  \cdot  V} \bigr) = \varepsilon.
\end{eqnarray*}
This function is again decreasing, convex, infinitely differentiable,
takes value $0$ for $\varepsilon=1$, and it satisfies $\kappa_{\mu
}(\varepsilon)\longrightarrow+ \infty$ when $\varepsilon\to0^+$.

Now, we let $m$ be a reference measure on $E$ s.t. $m  ( \{ V=0 \}
 ) =0$. We consider the sequence $(\beta_n)_n$ and its associated
Gibbs measures $\eta_n = \mu_{\beta_n} \propto \mathrm{e}^{-\beta_n V}. m$,
defined recursively by the equation
%
\begin{equation}
{\Delta}_{n+1}:= (\beta_{n+1} - \beta_{n}) =
\kappa_{\eta
_n}(\varepsilon). \label{increment-theo}
\end{equation}

 In an equivalent way, we have
\[
\lambda_{\eta_n}({\Delta}_{n+1}) = \varepsilon\quad\mbox{or}\quad
\eta_n \bigl(\mathrm{e}^{-{\Delta}_{n+1}  \cdot  V}\bigr) = \varepsilon.
\]

 The main objective of this section is to approximate these
target measures. Formally speaking, $(\eta_n)$ admits the FK structure
described in Section~\ref{section-Gibbs-motiv}, with potential
functions $G_n(x) = \mathrm{e}^{-\Delta_n.V(x)}$, and some dedicated MCMC
Markov kernels $M_n$. We let $g_n$, $b_n$, be the associated
oscillations, Dobrushin ergodic coefficients and the corresponding FK
transformations $\phi_n$.

The solving of the equation $ { \eta_n (\mathrm{e}^{-{\Delta}_{n+1}
 \cdot V}) = \varepsilon}$ can be interpreted as a way to impose some
kind of theoretical regularity in the FK flow. Indeed, according to the
formula (\ref{defgamma1}) (and definition~(\ref{def-mes-Gibbs})), it is
equivalent to find $\Delta_{n+1}$ s.t.
\[
\frac{\gamma_{n+1} (1) }{\gamma_n(1)} = \varepsilon\qquad \biggl( = \frac{Z_{\beta_n + \Delta_{n+1} }}{Z_{\beta_n}} \biggr).
\]
In other words, the sequence $\Delta_n$ is defined so that the
normalizing constants $\gamma_n(1)$ increase geometrically, with the
ratio $\gamma_{n+1} (1) / \gamma_n(1) = \varepsilon$. Notice that these
increments $\Delta_n$ are only theoretical, and the corresponding
potential functions $G_n$ are not explicitly known.


\subsection{An adaptive interacting particle model}
\label{description-algo-adapt}

As in the classic IPS algorithm, we approximate the measures $\eta_n$
by simulating an interacting particle system $(\zeta_n)_n =  ( \zeta
_n^1, \ldots, \zeta_n^N  )_n$ of size $N$ so that
\[
\eta^N_{n}=\frac{1}{N}\sum
_{1\leq i\leq N}\delta_{\zeta^i_n}\rightarrow _{N\uparrow\infty}
\eta_n.
\]

 We start with $N$ independent samples from $\eta_0$ and then
alternate selection and mutation steps, as described in Section~\ref{algo-classique}. As we mentioned above, in contrast to the classic IPS
model, the potential function $G_{n+1}$ arising in the selection is not
known. The selection step then starts by calculating the empirical
increment $\Delta_{n+1}^N$ defined by
\[
{\Delta}_{n+1}^N:= \kappa_{\eta_n^N}(\varepsilon)
\]
or $\lambda_{\eta_n^N}({\Delta}_{n+1}^N) = \eta_n^N
(\mathrm{e}^{-{\Delta}_{n+1}^N  \cdot  V}) = \varepsilon$. As the quantity
$ {\eta_n^N (\mathrm{e}^{-{\Delta}  \cdot  V}) = \frac{1}{N} \sum_{1
\leq i \leq N} \mathrm{e}^{-\Delta \cdot  V(\zeta_n^i)} } $ is easy to calculate for
all $\Delta\geq0$, one can calculate $\Delta_{n+1}^N$ by, for
example, performing a dichotomy algorithm. If we consider the
stochastic potential functions
\[
G_{n+1}^N = \mathrm{e}^{-{\Delta}_{n+1}^N.V}
\]
then every particle $\zeta^i_{n}$ evolves to a new particle
$\widehat{\zeta}^i_{n}$ randomly chosen with the following stochastic
selection transition:
\[
S_{n+1,\eta^N_{n}}^N\bigl( \zeta^i_{n},\mathrm{d}x
\bigr):=G_{n+1}^N\bigl( \zeta^i_{n}
\bigr) \delta_{ \zeta^i_{n}}(\mathrm{d}x)+ \bigl(1-G_{n+1}^N\bigl(
\zeta^i_{n}\bigr) \bigr) \Psi_{G_{n+1}^N}\bigl(
\eta^N_n\bigr) (\mathrm{d}x).
\]
In the above display formula, $\Psi_{G_{n+1}^N}(\eta^N_n)$ stands for
the updated measure defined by
\[
\Psi_{G_{n+1}^N}\bigl(\eta^N_n\bigr)=\sum
_{j=1}^N\frac{G_{n+1}^N( \zeta
^j_{n})}{\sum_{k=1}^N G_{n+1}^N( \zeta^k_{n})} \delta_{ \zeta^j_{n}}
.
\]
Note that $V_{\min}=0$ ensures $0 < G_{n+1}^N \leq1$.

 For the mutation step, we consider two models. The
implementable one consists in performing Markov transitions
$M_{n+1}^N(\widehat{\zeta}^i_{n}, \cdot)$, defined as $M_{n+1}(\widehat
{\zeta}^i_{n}, \cdot )$ by replacing $\beta_{n+1}$ by $\beta_{n+1}^N =
\beta_n^N + \Delta_{n+1}^N$. These kernels must be defined to let the
associated Boltzmann--Gibbs measures stable, for instance, using some
simulated annealing kernels (see Section~\ref{section-Gibbs-tuning},
page \pageref{section-Gibbs-tuning}). Thus, conditionally to the
previous particle set $\zeta_n$, the new population of particles $\zeta
_{n+1}$ is sampled from distribution
%
\begin{eqnarray}
\label{loi-algo-adapt} && \operatorname{Law} \bigl( \zeta_{n+1}^1,
\ldots,\zeta_{n+1}^N \mid\zeta _{n}^1,
\ldots,\zeta_{n}^N \bigr)
\nonumber
\\[-8pt]
\\[-8pt]
\nonumber
&&\quad= \bigl( \delta_{\zeta_{n}^1}.S^N_{n+1,\eta_{n}^N}.M_{n+1}^N
\bigr) \otimes\cdots\otimes \bigl( \delta_{\zeta
_{n}^N}.S^N_{n+1,\eta_{n}^N}.M_{n+1}^N
\bigr).
\end{eqnarray}

We also consider a simplified model where the mutation
transition is given by the limiting transition $M_{n+1}$ (MCMC Markov
kernel associated with theoretical temperature $\beta_{n+1}$ defined in
Section~\ref{FK-representation-adapt}):
%
\begin{eqnarray}
\label{loi-algo-adapt-simplif} && \operatorname{Law} \bigl( \zeta_{n+1}^1,
\ldots,\zeta_{n+1}^N \mid\zeta _{n}^1,
\ldots,\zeta_{n}^N \bigr)
\nonumber
\\[-8pt]
\\[-8pt]
\nonumber
&&\quad= \bigl( \delta_{\zeta_{n}^1}.S^N_{n+1,\eta_{n}^N}.M_{n+1}
\bigr) \otimes\cdots\otimes \bigl( \delta_{\zeta_{n}^N}.S^N_{n+1,\eta
_{n}^N}.M_{n+1}
\bigr).
\end{eqnarray}

The definition of $\Delta_{n+1}^N$ is to be interpreted as
the natural approximation of the theoretical relation (\ref
{increment-theo}). On the other hand, it admits a purely algorithmic
interpretation. As a matter of fact, conditionally to the nth
generation of particles $  ( \zeta_n^1, \ldots, \zeta_n^N  )$,
the probability for any particle $ \zeta_n^i $ to be accepted, that is,
not affected by the recycling mechanism, is given by $G_{n+1}^N(\zeta
_n^i) = \mathrm{e}^{-\Delta_{n+1}^N .  V(\zeta_n^i)}$. Then the expectation of
the number of accepted particles is given by $\sum_i \mathrm{e}^{-\Delta_{n+1}^N
. V(\zeta_n^i)}$. But it turns out that this quantity is exactly $N
\times\eta_n^N (\mathrm{e}^{-{\Delta}^{N}_{n+1}  \cdot  V})$, which is equal to
$N .  \varepsilon$ by definition of $\Delta_{n+1}^N$. Therefore,
$\varepsilon$ is an approximation of the proportion of particles which
remain in place during the selection step. In other words, at each
generation $n$, the increment $\Delta_{n+1}^N$ is chosen so that the
selection step kills less than $(1-\varepsilon).N$ particles. This type
of tuning parameter is very important in practice to avoid degenerate behaviors.

\subsection{A perturbation analysis}\label{sec43}
This section is mainly concerned with the convergence analysis of the
simplified adaptive model~(\ref{loi-algo-adapt-simplif}). The analysis
of the adaptive model (\ref{loi-algo-adapt}) is much more involved, and
our approach does not apply directly to study the convergence of this
model.

Despite the adaptation, the sequence $\eta_n^N$ can be analyzed as a
random perturbation of the theoretical sequence $\eta_n$. Let us fix
$n$ and a population state $\zeta_n$ at time $n$. If $\phi_{n+1}^N$
denotes the FK transformation associated with potential $G_{n+1}^N$ and
kernel $M_{n+1}$, then, by construction, the measure $\eta_{n+1}^N$ is
close to $\phi_{n+1}^N(\eta_{n}^N)$. In particular, by the Khintchine's
type inequalities presented in \cite{DM-FK} (see Lemma~7.3.3 page 223),
we have
%
\begin{equation}
\label{MZ-cond} \forall f \in\mathcal{B}_1 (E) \qquad\mathbb{E} \bigl(
\bigl\llvert \eta _{n+1}^N (f) - \phi_{n+1}^N
\bigl(\eta_{n}^N\bigr) (f) \bigr\rrvert ^p \mid
\zeta_n \bigr)^{1/p} \leq\frac{B_p}{\sqrt{N}},
\end{equation}
with the constants $B_p$ introduced in (\ref{def-Bp}). A
simple, but important remark about the Boltzmann--Gibbs transformations
is that for any measure $\mu$ and any positive functions $G$ and $\tilde
{G}$ we have
\[
\psi_{\tilde{G}} (\mu) = \psi_G \bigl( \psi_{{\tilde{G}}/{G}} (
\mu) \bigr).
\]
Therefore, if we take $ {H_{n+1}^N:= \frac
{G^N_{n+1}}{G_{n+1}} }$, then the perturbed transformation $\phi
_{n+1}^N$ can be written in terms of the theoretical one $\phi_{n+1}$ by
%
\begin{equation}
\label{mes-virt} \phi_{n+1}^N = \phi_{n+1} \circ
\psi_{H_{n+1}^N}.
\end{equation}

 If we use an inductive approach, we face the following
problem. Let $\eta$ be a deterministic measure ($\eta_n$ in our
analysis) and $\hat{\eta}$ a random measure ($\eta_n^N$ in our
analysis), close to $\eta$ under the $d_p$ distance (induction
hypothesis). We also consider a Markov kernel $M$ and the potential functions
%
\begin{equation}
\label{def-G-etc} G = \mathrm{e}^{-\kappa_{\eta}(\varepsilon).V}, \qquad \hat{G} = \mathrm{e}^{-\kappa_{\hat
{\eta}}(\varepsilon).V},\qquad  \hat{H} =
\frac{\hat{G}}{G}
\end{equation}
and we let $\phi$ (resp., $\hat{\phi}$) be the FK transformation
associated with the potential function $G$ (resp., $\hat{G}$).
The question is now: how can we estimate $d_p(\phi(\eta),\hat{\phi}(\hat
{\eta}))$ in terms of $d_p(\eta,\hat{\eta})$?

 To answer to this question, we propose to achieve a two-step
estimation. Firstly, we estimate the distance between $\hat{\eta}$ and
$\psi_{\hat{H}}(\hat{\eta}) $ (Lemma~\ref{lemme-fonda-adapt}).
Secondly, we analyze the stability properties of the transformation
$\phi$ (Lemma~\ref{lemme-phi-Lp}). This strategy is summarized by the
following synthetic diagram:
\[
\xymatrix@R=0.005cm{
 \eta \ar[r]_{\phi}
& \phi(\eta)
\\
\hat{\eta} &
\\
\downarrow&
\\
\psi_{\hat{H}}(\hat{\eta})  \ar[r]_{\phi}
& \hat{\phi}(\hat{\eta}). }
\]


\begin{lemma} \label{lemme-fonda-adapt}
Let $\eta\in\mathcal{P}(E)$, $\hat{\eta} \in\mathcal{P}_{\Omega}(E)$
and let $G$, $\hat{G}$, $\hat{H}$ be the positive functions on $E$
defined by the equations (\ref{def-G-etc}) for some $\varepsilon> 0$.
If $\eta( \{V=0 \}) = \hat{\eta}( \{V=0 \}) = 0$ (a.s.), then we have
\[
d_p \bigl( \psi_{\hat{H}}(\hat{\eta}), \hat{\eta} \bigr) \leq
\frac
{V_{\max} \cdot  \mathrm{e}^{\kappa_{\eta}(\varepsilon).V_{\max} } }{\varepsilon
 \cdot \eta(V)}  \cdot  d_p ( \hat{\eta}, \eta ).
\]
\end{lemma}

\begin{pf}
We simplify the notation and we set $x:=\kappa_{\eta}(\varepsilon)$ and
$\hat{x}:= \kappa_{\hat{\eta}}(\varepsilon)$.

We start with the following observation:
%
\begin{equation}
\label{prem-decomp} \psi_{\hat{H}} (\hat{\eta}) (f) - \hat{\eta} (f) =
\frac{\hat{\eta}
(\hat{H}.f )}{\hat{\eta}(\hat{H})} - \hat{\eta}(f) = \underbrace{\frac
{1}{\hat{\eta}(\hat{H})}}_{A_1}
\underbrace{\hat{\eta} \bigl[ \bigl(\hat {H} - \hat{\eta}(\hat{H}) \bigr).f
\bigr]}_{A_2}
\end{equation}
for any $f \in\mathcal{B}(E)$. We notice that $\hat{H} = \hat
{G}/G = \mathrm{e}^{(x-\hat{x})  \cdot  V}$, which leads to the lower bound $\hat
{\eta}(\hat{H}) = \hat{\eta}  ( \mathrm{e}^{(x-\hat{x})  \cdot V}  )
\geq\hat{\eta}  ( \mathrm{e}^{-\hat{x}  \cdot  V}  ) = \varepsilon$.
The last equality comes from the definition of $\hat{x}$. We just
proved: $|A_1| \leq\varepsilon^{-1}$. On the other hand, we have $
\operatorname{osc}(\hat{H}) = \llvert  \mathrm{e}^{(x-\hat{x})  \cdot  V_{\max}} - 1 \rrvert  $, so that
%
\begin{eqnarray}
\label{apparition-u} \llvert A_2 \rrvert \leq\hat{\eta} \bigl( \bigl
\llvert \hat{H} - \hat{\eta }(\hat{H}) \bigr\rrvert \bigr)  \cdot \llVert f \rrVert
_{\infty} & \leq& \operatorname{osc}(\hat{H})  \cdot \llVert f \rrVert _{\infty}
\nonumber
\\[-8pt]
\\[-8pt]
\nonumber
& \leq&\bigl\llvert \mathrm{e}^{(x-\hat{x})  \cdot  V_{\max}} - 1 \bigr\rrvert  \cdot\llVert f \rrVert
_{\infty}.
\end{eqnarray}

The quantity $ { \hat{u}:=  ( \mathrm{e}^{(x-\hat{x})
 \cdot  V_{\max}} - 1  ) }$ is intuitively small. Next, we provide
an estimate in terms of the functions $\lambda_{\eta}$ and $\lambda
_{\hat{\eta}}$. Given $\omega\in\Omega$, if $x \geq\hat{x}$, then we
can write
\[
\underbrace{\lambda_{\eta}(x)}_{=\varepsilon=\lambda_{\hat{\eta}}(\hat
{x})} - \lambda_{\hat{\eta}}(x)
= \lambda_{\hat{\eta}}(\hat{x}) - \lambda_{\hat{\eta}}(x) = \int
_{\hat{x}}^{x} -\lambda^{\prime}_{\hat{\eta}}
(s) \,\mathrm{d}s.
\]

 Furthermore, for all $\mu\in\mathcal{P}(E)$ and $s \geq0$,
we have
\begin{eqnarray*}-\lambda_{\mu}^{\prime}(s) &=& \mu
\bigl( V  \cdot  \mathrm{e}^{-sV} \bigr)  \geq \mu \bigl( V  \cdot  \mathrm{e}^{-s V_{\max}}
\bigr)
\\
& \geq& \mu(V)  \cdot  \mathrm{e}^{-s V_{\max}}.
\end{eqnarray*}

 Then we have
\begin{eqnarray*}
\lambda_{\eta}(x) - \lambda_{\hat{\eta}}(x) & \geq&\hat{\eta}(V)
\frac
{-1}{V_{\max}} \bigl[ \mathrm{e}^{-s V_{\max}} \bigr]_{\hat{x}}^{x}
\\
& =& \hat{\eta}(V) \frac{\mathrm{e}^{-x V_{\max}} }{V_{\max}} \bigl( \mathrm{e}^{(x-\hat
{x})V_{\max}} - 1 \bigr)
\\
& = &\hat{\eta}(V) \frac{\mathrm{e}^{-x V_{\max}} }{V_{\max}} \hat{u}.
\end{eqnarray*}
By symmetry, we have
\[
x \leq\hat{x} \quad\Longrightarrow\quad\lambda_{\hat{\eta}}(x) - \lambda
_{\eta}(x) \geq\hat{\eta}(V) \frac{\mathrm{e}^{-x V_{\max}} }{V_{\max}} (-\hat {u}).
\]
This yields the almost sure upper bound
\[
\llvert \hat{u} \rrvert \cdot \hat{\eta}(V) \frac{\mathrm{e}^{-x V_{\max}}
}{V_{\max}} \leq\bigl\llvert
\lambda_{\eta}(x) - \lambda_{\hat{\eta}}(x) \bigr\rrvert.
\]

Using the decomposition $\hat{\eta}(V) = \eta(V) + (\hat{\eta
}(V) - \eta(V) )$, by simple manipulation, we prove that
\[
\llvert \hat{u} \rrvert \leq\frac{V_{\max} \mathrm{e}^{x V_{\max}} }{\eta(V)} \underbrace{\bigl\llvert
\lambda_{\eta}(x) - \lambda_{\hat{\eta}}(x) \bigr\rrvert
}_{A_3} + \underbrace{\frac{\llvert  \hat{\eta}(V) - \eta(V) \rrvert
}{\eta(V)}}_{A_4}  \cdot
\underbrace{\llvert \hat{u} \rrvert }_{A_5}.
\]

 Considering the $L^p$ norm of the right-hand side of this
inequality, one can check that
\begin{itemize}
\item$A_3 = (\eta- \hat{\eta})  ( \mathrm{e}^{-x  \cdot  V}  ) $ so, as
$\operatorname{osc}  ( \mathrm{e}^{-x \cdot  V}  ) \leq1$, $\Vert A_3 \Vert_p \leq
d_p(\hat{\eta},\eta)$;
\item as $\operatorname{osc}(V) = V_{\max}$, $\Vert A_4 \Vert_p \leq\frac{V_{\max
}}{\eta(V)}  \cdot  d_p(\hat{\eta},\eta)$;
\item$A_5 = \llvert  \hat{u} \rrvert  = \mathrm{e}^{x V_{\max}} \llvert  \mathrm{e}^{-\hat{x}
V_{\max}} - \mathrm{e}^{-x V_{\max}} \rrvert  \leq \mathrm{e}^{x V_{\max}} $.
\end{itemize}
Making the appropriate substitutions, we have
\[
\Vert\hat{u} \Vert_p \leq\frac{2 V_{\max} \mathrm{e}^{x  \cdot  V_{\max}} }{\eta
(V)} \cdot  d_p(
\hat{\eta},\eta).
\]

 Combining this result with (\ref{prem-decomp}) and (\ref
{apparition-u}), we check that
\[
\bigl\llVert \psi_{\hat{H}} (\hat{\eta}) (f) - \hat{\eta} (f) \bigr\rrVert
_p \leq\frac{2 V_{\max} \mathrm{e}^{x V_{\max}} }{\varepsilon \cdot \eta(V)}  \cdot  d_p(\hat{\eta},\eta)
\cdot \llVert f \rrVert _{\infty}.
\]

 We finally go from $\llVert  f \rrVert _{\infty}$ to $\frac
{\operatorname{osc}(f)}{2}$ by noticing that $\psi_{\hat{H}} (\hat{\eta}) (f) - \hat
{\eta} (f)$ is equal to $0$ for any constant function $f$, and by
considering the above inequality taken for $\tilde{f}=f-\frac{f_{\max
}+f_{\min}}{2}$, which satisfies $\llVert  \tilde{f} \rrVert _{\infty} =
\frac{\operatorname{osc}(f)}{2}$.

This ends the proof of the lemma.
\end{pf}


\begin{lemma} \label{lemme-phi-Lp}
Let $\eta\in\mathcal{P}(E)$, $\hat{\eta} \in\mathcal{P}_{\Omega
}(E)$, and let $\phi$ be a FK transformation associated with a positive
function $G$ and a Markov kernel $M$. If we set $ { g:=
\sup_{x,y \in E} G(x) / G(y) }$ and $b:= \beta(M)$, then we have
\[
d_p \bigl( \phi(\hat{\eta}), \phi(\eta) \bigr) \leq g  \cdot  b  \cdot
d_p ( \hat{\eta}, \eta ).
\]
\end{lemma}

\begin{pf}
Let us fix $f \in\mathcal{B}(E)$. We have
\begin{eqnarray*}
\phi(\hat{\eta}) (f) - \phi(\eta) (f) & =& \frac{\hat{\eta}(G \times M.
f)}{\hat{\eta}(G)} - \phi(\eta) (f)
\\
& =& \frac{\hat{\eta}  ( G \times [ M.  ( f-\phi(\eta) (f)
 )  ]  )}{\hat{\eta}  ( G  )}.
\end{eqnarray*}
Let $ { \tilde{f} = M .   ( f-\phi(\eta) (f)  ) }
$. By property (\ref{prop-Dob}), $\tilde{f}$ satisfies $\operatorname{osc}(\tilde{f})
= \operatorname{osc}(M.f) \leq b  \cdot  \operatorname{osc}(f)$.

Additionally we have $\eta(G \times\tilde{f}) = \eta(G \times M.f) -
\eta ( G \times\frac{\eta(G \times M.f)}{\eta(G)}  ) = 0 $.
So we obtain
\begin{eqnarray*}
\phi(\hat{\eta}) (f) - \phi(\eta) (f)&= &\frac{\hat{\eta}(G \times\tilde
{f}) }{\hat{\eta}(G)} - \underbrace{
\frac{\eta(G \times\tilde{f})
}{\hat{\eta}(G)}}_{=0}  = \frac{1}{\hat{\eta}(G)} (\hat{\eta} - \eta)
(G \times\tilde{f})
\\
& = &\frac{G_{\max}}{\hat{\eta}(G)} (\hat{\eta} - \eta) \biggl(\frac
{G}{G_{\max}} \times
\tilde{f}\biggr).
\end{eqnarray*}

 Firstly, we notice that $ { \llvert  \frac{G_{\max
}}{\hat{\eta}(G)} \rrvert  \leq\frac{G_{\max}}{G_{\min}} \leq g }$. On
the other hand, we notice that $\tilde{f}$ can be rewritten as
\[
\tilde{f} = M. \bigl( f - \psi_G(\eta) (M.f) \bigr) = (M.f) - \psi
_G(\eta) (M.f).
\]

 It follows that $ \tilde{f}_{\max} \geq0$ and $ \tilde
{f}_{\min} \leq0$. Under these conditions, $\operatorname{osc}  ( \frac{G}{G_{\max
}} \times\tilde{f}  ) \leq \operatorname{osc}(\tilde{f}) \leq b  \cdot  \operatorname{osc}(f) $.
We conclude that
\begin{eqnarray*}
\biggl( \mathbb{E} \biggl\llvert \frac{G_{\max}}{\hat{\eta}(G)} (\hat{\eta}
- \eta) \biggl(
\frac{G}{G_{\max}} \times\tilde{f} \biggr) \biggr\rrvert ^p
\biggr)^{1/p} & \leq& g  \cdot \mathbb{E} \biggl[ \biggl\llvert (\hat{\eta}
- \eta ) \biggl( \frac{G}{G_{\max}} \times\tilde{f} \biggr) \biggr\rrvert
^p \biggr]^{1/p}
\\
& \leq& g  \cdot  b  \cdot d_p ( \hat{\eta}, \eta )  \cdot  \operatorname{osc}(f).
\end{eqnarray*}

 This ends the proof of the lemma.
\end{pf}

\begin{remark} \label{remarque-lemme-phi-Lp}
Lemma~\ref{lemme-phi-Lp} holds in the case where $\eta$ is also a
random measure (we now note $\tilde{\eta} \in\mathcal{P}_{\Omega}(E)$)
if one can find a $\sigma$-algebra $\mathcal{F}$ such that:
\begin{longlist}[1.]
\item[1.]$\tilde{\eta}$ is $\mathcal{F}$-measurable;
\item[2.] the ($\mathcal{F}$-measurable) random variable
\[
d_p^{\mathcal{F}} (\hat{\eta},\tilde{\eta}):=
\sup_{\tilde{f} \in
\mathcal{O}_1^{\Omega,\mathcal{F}}(E)} \mathbb{E} \bigl[ \bigl\llvert \hat{\eta} (\tilde{f}) -
\tilde{\eta} (\tilde{f}) \bigr\rrvert ^p \mid \mathcal{F}
\bigr]^{1/p}
\]
is uniformly bounded on $\Omega$. In the above definition, $\mathcal
{O}_1^{\Omega,\mathcal{F}}(E)$ denotes the set of random, $\mathcal
{F}$-measurable functions $\tilde{f} \dvtx E \rightarrow\mathbb{R}$
satisfying $\operatorname{osc}(\tilde{f}) \leq1$ a.s.
\end{longlist}
More precisely, under these conditions, we have
\[
d_p \bigl( \phi(\hat{\eta}), \phi(\tilde{\eta}) \bigr) \leq g  \cdot
b  \cdot \bigl\Vert d_p^{\mathcal{F}} ( \hat{\eta}, \tilde{\eta} )
\bigr\Vert_{\infty}.
\]
\end{remark}

\begin{pf}
We fix $f \in\mathcal{B}(E)$ and use the same line of arguments as in
the proof of Lemma~\ref{lemme-phi-Lp}. Since the function $
{ \tilde{f} = M .   ( f-\phi(\eta) (f)  ) } $ is $\mathcal
{F}$-measurable and satisfies
\[
\operatorname{osc} \biggl( \frac{G}{G_{\max}} \times\tilde{f} \biggr) \leq \operatorname{osc}(\tilde {f})
\leq b  \cdot  \operatorname{osc}(f) \qquad\mbox{a.s.},
\]
we have
\begin{eqnarray*}
&&\biggl( \mathbb{E} \biggl\llvert \frac{G_{\max}}{\hat{\eta}(G)} (\hat{\eta} - \tilde{\eta})
\biggl( \frac{G}{G_{\max}} \times\tilde{f} \biggr) \biggr\rrvert ^p
\biggr)^{1/p} \\
&&\quad \leq g  \cdot \mathbb{E} \biggl( \mathbb{E} \biggl[ \biggl
\llvert (\hat{\eta} - \tilde{\eta}) \biggl( \frac{G}{G_{\max}} \times \tilde{f}
\biggr) \biggr\rrvert ^p \Big\mid\mathcal{F} \biggr] \biggr)^{1/p}
\\
&&\quad \leq g  \cdot  b  \cdot  \operatorname{osc}(f)  \cdot \mathbb{E} \biggl( \mathbb{E} \biggl[ \biggl
\llvert (\hat{\eta} - \tilde{\eta}) \biggl( \frac{ {G}/{G_{\max
}} \times\tilde{f} }{b \cdot  \operatorname{osc}(f)} \biggr) \biggr
\rrvert ^p \Big\mid \mathcal{F} \biggr] \biggr)^{1/p}
\\
&&\quad\leq g  \cdot  b  \cdot  \operatorname{osc}(f)  \cdot \bigl\Vert d_p^{\mathcal{F}}
( \hat{
\eta}, \tilde{\eta} ) \bigr\Vert_{\infty}.
\end{eqnarray*}
This ends the proof.
\end{pf}

\subsection{Nonasymptotic convergence results} \label{gros-result-adapt}

This section is mainly concerned with the proof of Theorem~\ref
{theo-statement-adapt} stated on page \pageref{theo-statement-adapt}.
We also deduce some concentration inequalities of the ISA adaptive
model. In this section, $(\eta_n)$ denotes the sequence of theoretical
measures defined in Section~\ref{FK-representation-adapt}, and $(\eta
_n^N)$ denotes the sequence of empirical measures associated with the
particle system described in (\ref{loi-algo-adapt-simplif}),
page \pageref{loi-algo-adapt-simplif}. We start with the proof of
Theorem~\ref{theo-statement-adapt}.

\begin{pf*}{Proof of Theorem~\ref{theo-statement-adapt}}
We fix $p \geq1$ and we let $ \tilde{e}_n = \sum_{k=0}^n \prod_{i=k+1}^n b_i g_i (1+c_i) $. We notice that this sequence can also be
defined with the recurrence relation $\tilde{e}_{n+1} = 1+g_{n+1}
b_{n+1}(1+c_{n+1})  \cdot \tilde{e}_{n}$ starting at $\tilde{e}_0 = 1$.
We also consider the following parameter:
\[
e_n:= \frac{2\sqrt{N}}{ B_p}  \cdot \sup_{f \in\mathcal{O}_1
(E)} \bigl\llVert
\eta_n^N (f) - \eta_n (f) \bigr\rrVert
_p.
\]
We use an inductive proof to check that the proposition $ {\textbf
{IH}(n) = \{ e_n \leq\tilde{e}_n \} }$ is met at any rank~$n$. As $\eta
_0^N$ is obtained with $N$ independent samples from $\eta_0$, $\textbf
{IH}(0)$ is given by the Khintchine's inequality (see the relation (\ref
{remarque-osc-born1}), page \pageref{remarque-osc-born1} for the factor
$1/2$). Now suppose that $\textbf{IH}(n)$ is satisfied.

According to the identity (\ref{mes-virt}), we can write the following
decomposition:
%
\begin{eqnarray}
\label{decompAA} \eta_{n+1}^N - \eta_{n+1} & =&
\bigl( \eta_{n+1}^N - \phi_{n+1}^N
\bigl(\eta _{n}^N\bigr) \bigr) + \bigl(
\phi_{n+1}^N \bigl(\eta_{n}^N\bigr) -
\eta_{n+1} \bigr)
\nonumber
\\[-8pt]
\\[-8pt]
\nonumber
& =& \underbrace{ \bigl( \eta_{n+1}^N -
\phi_{n+1}^N \bigl(\eta_{n}^N\bigr)
\bigr)}_{A_1} + \underbrace{ \bigl( \phi_{n+1} \bigl(
\psi_{H_{n+1}^N} \bigl(\eta _{n}^N\bigr) \bigr) -
\phi_{n+1}(\eta_{n}) \bigr)}_{A_2}.
\end{eqnarray}
Using (\ref{MZ-cond}), page \pageref{MZ-cond}, we know that for all
function $f \in\mathcal{O}_1 (E)$, we have
\begin{eqnarray*}
\bigl\llVert A_1 (f) \bigr\rrVert _p & =& \mathbb{E}
\bigl( \mathbb{E} \bigl[ \bigl\llvert \eta_{n+1}^N (f) -
\phi_{n+1}^N \bigl(\eta_{n}^N\bigr)
(f) \bigr\rrvert ^p \mid\zeta_n \bigr]
\bigr)^{1/p}
\\
& \leq&\frac{B_p}{2 \sqrt{N}}
\end{eqnarray*}
(see the relation (\ref{remarque-osc-born1}), page \pageref
{remarque-osc-born1} for the factor $1/2$), so $ { \frac
{2\sqrt{N}}{ B_p} \llVert  A_1 (f) \rrVert _p \leq1 } $.

 To estimate $A_2$, we start by decomposing $  ( \psi
_{H_{n+1}^N} (\eta_{n}^N) - \eta_{n}  )$ in this way:
\[
\psi_{H_{n+1}^N} \bigl(\eta_{n}^N\bigr) -
\eta_{n} = \underbrace{ \bigl( \psi _{H_{n+1}^N} \bigl(
\eta_{n}^N\bigr) - \eta_{n}^N \bigr)
}_{Q_1} + \underbrace { \bigl( \eta_{n}^N -
\eta_{n} \bigr)}_{Q_2}.
\]
 By the induction hypothesis, we have $\frac{2\sqrt{N}}{ B_p}
 \cdot \sup_{f \in\mathcal{O}_1 (E)} \llVert  Q_2 (f) \rrVert _p \leq\tilde{e}_{n}$.
 Therefore, by Lem\-ma~\ref
{lemme-fonda-adapt}, we find that
\[
\frac{2\sqrt{N}}{B_p}  \cdot \sup_{f \in\mathcal{O}_1 (E)} \bigl\llVert Q_1 (f)
\bigr\rrVert _p \leq c_{n+1}  \cdot \tilde{e}_{n}
.
\]

 Thus, the measures $\psi_{H_{n+1}^N} (\eta_{n}^N)$ and $\eta
_{n}$ satisfy
\[
\frac{2\sqrt{N}}{ B_p}  \cdot\sup_{f \in\mathcal{O}_1 (E)} \bigl\llVert
\psi_{H_{n+1}^N} \bigl(\eta_{n}^N\bigr) (f) -
\eta_{n} (f) \bigr\rrVert _p \leq(1+c_{n+1})
 \cdot \tilde{e}_{n}.
\]

 Applying Lemma~\ref{lemme-phi-Lp}, we also have
\[
\frac{2\sqrt{N}}{B_p}  \cdot \sup_{f \in\mathcal{O}_1 (E)} \bigl\llVert A_2 (f)
\bigr\rrVert _p \leq g_{n+1} b_{n+1}
(1+c_{n+1})  \cdot  \tilde{e}_{n}.
\]

Back to (\ref{decompAA}), we conclude that $ e_{n+1} \leq1+
g_{n+1} b_{n+1} (1+c_{n+1})  \cdot \tilde{e}_{n} = \tilde{e}_{n+1}$.

This ends the proof of the theorem.
\end{pf*}

We are now in position to obtain a sufficient condition for uniform
concentration and $L^p$-mean error bounds w.r.t. time for the
simplified adaptive
model discussed in Sections~\ref{description-algo-adapt} and \ref{sec43}.


\begin{corollary}
If the condition $b_n g_n (1+c_n) \leq a$ is satisfied for some $a < 1$
and any $n$, then we have the uniform error bounds
%
\begin{equation}
\label{Lp-unif-adapt} d_p \bigl( \eta_n^N,
\eta_n \bigr) \leq\frac{B_p}{2(1-a)\sqrt{N}}
\end{equation}
for any $p$, with the constants $B_p$ introduced in (\ref
{def-Bp}). In addition, for any $f \in\mathcal{B}_1(E)$, we have the
following concentration inequalities:
%
\begin{equation}
\label{conc-adapt-un} \forall s \geq0\qquad \mathbb{P} \bigl( \bigl\vert\eta_n^N(f)
- \eta _n(f) \bigr\vert\geq s \bigr) \leq r_1 (\sqrt{N}.s)
\mathrm{e}^{-r_2 N s^2}
\end{equation}
and
%
\begin{equation}
\label{conc-adapt-deux} \forall y \geq0\qquad \mathbb{P} \biggl( \bigl\vert\eta_n^N(f)
- \eta _n(f) \bigr\vert\geq\frac{r(1+\sqrt{2y})}{\sqrt{N}} \biggr) \leq
\mathrm{e}^{-y},
\end{equation}
with the parameters
\[
\cases{ %
r_1= { \mathrm{e}^{1/2} (1-a) },
\vspace*{2pt}\cr
r_2={\displaystyle \frac{1}{2} (1-a)^2 },
\vspace*{2pt}\cr
\displaystyle r = { \frac{1}{1-a} }. }
\]
\end{corollary}

\begin{pf}
The inequality (\ref{Lp-unif-adapt}) is a direct consequence of
Theorem~\ref{theo-statement-adapt} (see the relation (\ref{remarque-osc-born1}),
page~\pageref{remarque-osc-born1} for the factor $1/2$).

Let us fix $n$, $f \in\mathcal{B}_1(E)$ and set
\[
X:= \bigl\vert\eta_n^N(f) - \eta_n(f) \bigr\vert
\quad\mbox{and}\quad \epsilon_N:= \frac{1}{(1-a) \sqrt{N}}.
\]

 In this notation, we have $ { \llVert  X \rrVert _p \leq B_p  \cdot\epsilon_N } $ for any $p \geq1$. Let us fix
$s\geq0$. By Markov inequality, for all $t\geq0$ we have
%
\begin{equation}
\label{ineg-markov-expo} \mathbb{P} ( X \geq s ) = \mathbb{P} \bigl( \mathrm{e}^{t X}
\geq \mathrm{e}^{t s} \bigr) \leq \mathrm{e}^{-st} \mathbb{E} \bigl(
\mathrm{e}^{t X} \bigr).
\end{equation}

 Using the formulae (\ref{def-Bp}), page \pageref{def-Bp}, we
estimate the Laplace transform $ { \mathbb{E}  ( \mathrm{e}^{t
X}  ) }$.
\begin{eqnarray*}
\mathbb{E} \bigl( \mathrm{e}^{t X} \bigr) & =& \sum
_{p \geq0} \mathbb{E} \biggl( \frac{t^p .  X^p}{p!} \biggr)
\\
& \leq&\sum_{p \geq0} \frac{t^{2p} \epsilon_N^{2p} }{(2p)!}
\frac
{(2p)!}{2^p. p!} + \sum_{p \geq0} \frac{t^{2p+1} \epsilon_N^{2p+1}
}{(2p+1)!}
\frac{(2p+1)!}{2^p. p! \sqrt{2p+1}}
\\
& \leq&(1+ t \epsilon_N) \mathrm{e}^{{t^2 \epsilon_N^2}/{2}}.
\end{eqnarray*}

 Taking the inequality (\ref{ineg-markov-expo}) with
$ { t= \frac{1}{\epsilon_N}  ( \frac{s}{\epsilon_N} -1
 ) }$, we obtain
\[
\mathbb{P} ( X \geq s ) \leq\frac{s}{\epsilon_N} \mathrm{e}^{-
({1}/{2} ) [  ({s}/{\epsilon_N}  )^2 -1  ] },
\]
which is equivalent to the first concentration
inequality (\ref{conc-adapt-un}).

For the second one, let $u=\frac{s}{\epsilon_N}$. Since $\log(u) \leq
u-1$ we have
\[
\mathbb{P} ( X \geq s ) = \mathbb{P} ( X \geq \epsilon_N  \cdot  u )
\leq u  \cdot  \mathrm{e}^{- ({u^2}/{2} -
{1}/{2} ) } = \mathrm{e}^{-({u^2}/{2} - \log u - {1}/{2})}
\leq \mathrm{e}^{-({u^2}/{2} - u + {1}/{2})}.
\]
Let $T$ the function defined by
\begin{eqnarray*} && [1,+\infty)  \longrightarrow [0,+\infty),
\\
&& T =  v  \mapsto \frac{v^2}{2} - v + \frac{1}{2}
\end{eqnarray*}
$T$ is bijective, and its inverse function is $T^{-1}\dvtx y \mapsto1 +
\sqrt{2y}$. Thus, we have
\[
\forall y \geq0\qquad  \mathbb{P} \bigl( X \geq\epsilon_N  \cdot
T^{-1}(y) \bigr) \leq \mathrm{e}^{-y},
\]
which is equivalent to (\ref{conc-adapt-deux}).

This ends the proof of the corollary.
\end{pf}


\begin{appendix}
\section*{Appendix}\label{app}

\setcounter{subsection}{0}
\subsection{Proof of Lemma \texorpdfstring{\protect\ref{BG-tv}}{6}, 
page \texorpdfstring{\protect\pageref{BG-tv}}{8}}
\label{preuve-BG-tv}

For all $\eta\in\mathcal{P}(E)$, let $S_{\eta}$ be the Markov kernel
defined by
\[
S_{\eta}(x,\mathrm{d}y) = \frac{G(x)}{G_{\max}} \delta_x(\mathrm{d}y) +
\biggl( 1-\frac
{G(x)}{G_{\max}} \biggr) \Psi_G(\eta) (\mathrm{d}y).
\]
This kernel satisfies $ \eta.S_{\eta} = \Psi_G(\eta)$, since for all $f
\in\mathcal{B}(E)$ we have
\begin{eqnarray*}
\eta.S_{\eta} (f) & =& \frac{\eta(G \times f)}{G_{\max}} + \biggl( 1-
\frac{\eta(G)}{G_{\max}} \biggr) .  \underbrace{ \Psi_G(\eta) (f)
}_{={\eta(G \times f)}/{\eta(G)}}
\\
& = &\frac{\eta(G \times f)}{G_{\max}} + \frac{\eta(G \times f)}{\eta
(G)} - \frac{\eta(G).\eta(G \times f)}{G_{\max}.\eta(G)}
\\
& =& \Psi_G(\eta) (f).
\end{eqnarray*}
We use the decomposition
\[
\Psi_G(\mu) - \Psi_G(\nu) = \mu.S_{\mu} -
\nu.S_{\nu} = (\mu- \nu ).S_{\mu} + \nu.(S_{\mu}-S_{\nu})
.
\]
On the one hand, $ \llVert (\mu- \nu).S_{\mu} \rrVert _{\mathrm{tv}} \leq\llVert  \mu- \nu\rrVert _{\mathrm{tv}} .  \beta(S_{\mu}) \leq\llVert  \mu- \nu
\rrVert _{\mathrm{tv}} $.

 On the other hand, for all nonnegative function $f \in
\mathcal{B}_1(E)$,
\[
(S_{\mu}-S_{\nu}) (f) (x) = \underbrace{ \biggl( 1-
\frac{G(x)}{G_{\max}} \biggr) }_{ \leq1-{1}/{g}} \underbrace{ \bigl[
\Psi_G(\mu) (f) - \Psi_G(\nu) (f) \bigr]
}_{| - | \leq\llVert  \Psi_G(\mu) - \Psi
_G(\nu) \rrVert _{\mathrm{tv}}}
\]
so $ \llvert  \nu. (S_{\mu}-S_{\nu})(f) \rrvert  \leq ( 1-\frac
{1}{g}  ) \llVert  \Psi_G(\mu) - \Psi_G(\nu) \rrVert _{\mathrm{tv}} $.

Thus,
\begin{eqnarray*}
&&\bigl\llVert \Psi_G(\mu) - \Psi_G(\nu) \bigr\rrVert
_{\mathrm{tv}} \leq \llVert \mu - \nu\rrVert _{\mathrm{tv}} + \biggl( 1-
\frac{1}{g} \biggr) \bigl\llVert \Psi_G(\mu) -
\Psi_G(\nu) \bigr\rrVert _{\mathrm{tv}}
\\
&&\quad\Rightarrow\quad \bigl\llVert \Psi_G(\mu) - \Psi_G(\nu)
\bigr\rrVert _{\mathrm{tv}}  \leq g.\llVert \mu- \nu\rrVert _{\mathrm{tv}}.
\end{eqnarray*}
This ends the proof of the lemma.

\subsection{Proof of the $L^p$-mean error bound 
\texorpdfstring{(\protect\ref{born-Lp-DMM})}{(1.10)}
(page \texorpdfstring{\protect\pageref{born-Lp-DMM}}{12})}

\label{preuve-born-Lp-DMM}

We fix $p \geq1$. Let us start with this remark: for all $\hat{\mu},
\hat{\nu} \in\mathcal{P}_{\Omega} (E)$, $\hat{\mu}(f) - \hat{\nu}(f) =
0$ for all constant function $f$, so we have
%
\setcounter{equation}{0}
\begin{eqnarray}\label
{remarque-osc-born1}
\sup _{\Vert f \Vert_{\infty} \leq1}\bigl\Vert\hat{\mu}(f) - \hat {\nu}(f) \bigr\Vert_{p} &
= &2 \sup_{\Vert f \Vert_{\infty} \leq
{1}/{2}} \bigl\Vert\hat{\mu}(f) - \hat{\nu}(f)\bigr \Vert_{p}\nonumber
\\
&= & 2 \sup_{\operatorname{osc}(f) = 1}\bigl \Vert\hat{\mu}(f) - \hat{\nu}(f)\bigr \Vert
_{p}
\\
& =&2  \cdot  d_p (\hat{\mu}, \hat{\nu}).\nonumber
\end{eqnarray}
In the context of the interacting particle system described in
Section~\ref{algo-classique}, page \pageref{algo-classique}, we use the
decomposition
%
\begin{equation}
\label{decomp-classique} \eta_n^N - \eta_n = \sum
_{k=0}^n \phi_{k,n} \bigl(
\eta_k^N \bigr) - \phi_{k,n} \bigl(
\phi_k \bigl( \eta_{k-1}^N \bigr) \bigr)
\end{equation}
(with the convention $\phi_0(\eta_{-1}^N) = \eta_0$). For all $k\in\{
1,2,\ldots,n \}$, let $\mathcal{F}_{k-1}$ be the $\sigma$-algebra
generated by the variable $\zeta_{k-1}$ (we set $\mathcal{F}_{0} = \{
\varnothing, E \}$). Remark~\ref{remarque-lemme-phi-Lp}, page \pageref
{remarque-lemme-phi-Lp}, implies
%
\begin{equation}
\label{passage-phi-p-n-cond} d_p \bigl( \phi_{k,n} \bigl(
\eta_k^N \bigr), \phi_{k,n} \bigl( \phi
_k \bigl( \eta_{k-1}^N \bigr) \bigr) \bigr) \leq
g_{k,n}  \cdot  b_{k,n}  \cdot  \bigl\llVert d_p^{\mathcal{F}_{k-1}}
\bigl( \eta_k^N, \phi_k \bigl(
\eta_{k-1}^N \bigr) \bigr) \bigr\rrVert _{\infty},
\end{equation}
where $d_p^{\mathcal{F}_{k-1}}  ( \eta_k^N, \phi_k ( \eta_{k-1}^N )
 )$ is the random variable defined by
\[
d_p^{\mathcal{F}_{k-1}} \bigl( \eta_k^N,
\phi_k \bigl( \eta_{k-1}^N \bigr) \bigr):=
\sup_{\tilde{f} \in\mathcal{O}_1^{\Omega,\mathcal
{F}_{k-1}}(E)} \mathbb{E} \bigl[ \bigl\llvert \eta_k^N
(\tilde{f}) - \phi_k \bigl( \eta_{k-1}^N \bigr)
(\tilde{f}) \bigr\rrvert ^p \mid\mathcal {F}_{k-1}
\bigr]^{1/p}.
\]
In the above definition, $\mathcal{O}_1^{\Omega,\mathcal{F}_{k-1}}(E)$
denotes the set of random, $\mathcal{F}_{k-1}$-measurable functions
$\tilde{f} \dvtx E \rightarrow\mathbb{R}$ satisfying $\operatorname{osc}(\tilde{f}) \leq
1$ a.s.

 Given the conditioning relation
\begin{eqnarray*}
&& \operatorname{Law} \bigl( \zeta_{k}^1,\ldots,
\zeta_{k}^N \mid\zeta _{k-1}^1,
\ldots,\zeta_{k-1}^N \bigr)
\\
&&\quad= ( \delta_{\zeta_{k-1}^1}.S_{k,\eta_{k-1}^N}.M_{k} ) \otimes\cdots
\otimes ( \delta_{\zeta_{k-1}^N}.S_{k,\eta
_{k-1}^N}.M_{k} )
\end{eqnarray*}
and the Khintchine's type inequalities presented in \cite{DM-FK}, we
have for all random, $\mathcal{F}_{k-1}$-measurable function $\tilde
{f}$ satisfying $ \Vert\tilde{f} \Vert_{\infty} \leq1$ a.s.:
\[
\mathbb{E} \bigl( \bigl\llvert \eta_{k}^N (\tilde{f}) -
\phi_{k}\bigl(\eta _{k-1}^N\bigr) (\tilde{f})
\bigr\rrvert ^p \mid\mathcal{F}_{k-1} \bigr)^{1/p}
\leq\frac{B_p}{\sqrt{N}},
\]
with the constants $B_p$ introduced in (\ref{def-Bp}),
page \pageref{def-Bp}. This is equivalent to
\[
d_p^{\mathcal{F}_{k-1}} \bigl( \eta_k^N,
\phi_k \bigl( \eta_{k-1}^N \bigr) \bigr) \leq
\frac{B_p}{2\sqrt{N}}.
\]
Finally, we combine this result with (\ref{passage-phi-p-n-cond}) and
the decomposition (\ref{decomp-classique}) to obtain
\[
d_p\bigl(\eta_n^N, \eta_n\bigr)
\leq\frac{B_p}{2\sqrt{N}} \sum_{k=0}^n
g_{k,n} b_{k,n}.
\]
This ends the proof.
\end{appendix}



%




\printhistory
\end{document}